\documentclass[10pt]{amsart}

\usepackage{amssymb,amsmath,amsthm}
\usepackage[all]{xy}
\usepackage{eucal}

\theoremstyle{plain}
\newtheorem{thm}[subsection]{Theorem}

\newtheorem{prop}[subsection]{Proposition}

\theoremstyle{definition}

\theoremstyle{remark}
\newtheorem{rem}[subsection]{Remark}

\newcommand{\ZZ}{{ \mathbb{Z} }}
\newcommand{\NN}{{ \mathbb{N} }}

\newcommand{\capC}{{ \mathcal{C} }}
\newcommand{\capO}{{ \mathcal{O} }}

\newcommand{\capR}{{ \mathcal{R} }}
\newcommand{\capX}{{ \mathcal{X} }}

\newcommand{\id}{{ \mathrm{id} }}

\newcommand{\ModR}{{ \mathsf{Mod}_\capR }}

\newcommand{\Spectra}{{ \mathsf{Sp}^\Sigma }}
\newcommand{\SpectraN}{{ \mathsf{Sp}^\NN }}

\newcommand{\Alg}{{ \mathsf{Alg} }}

\newcommand{\AlgO}{{ \Alg_\capO }}
\newcommand{\AlgOY}{{ \Alg^Y_\capO }}

\newcommand{\Loop}{{ \Omega }}
\newcommand{\Loopt}{{ \tilde{\Omega} }}
\newcommand{\Susp}{{ \Sigma }}
\newcommand{\Suspt}{{ \tilde{\Sigma} }}

\newcommand{\tensor}{{ \otimes }}

\newcommand{\wequiv}{{ \ \simeq \ }}

\newcommand{\Iso}{{  \ \cong \ }}

\newcommand{\rarrow}{{ \rightarrow }}

\newcommand{\function}[3]{{ {#1}\colon\thinspace{#2}\rarrow{#3} }}

\DeclareMathOperator*{\holim}{holim}

\begin{document}

\title[Functor calculus completions]{Functor calculus completions for retractive operadic algebras in spectra}

\author{Matthew B. Carr}
\author{John E. Harper}

\address{Department of Population and Public Health Sciences, University of Southern California, 
1450 Biggy St, NRT G517, Los Angeles, CA 90033}
\email{mbcarr@usc.edu}

\address{Department of Mathematics, The Ohio State University, Newark, 1179 University Dr, Newark, OH 43055, USA}
\email{harper.903@math.osu.edu}

\begin{abstract}
The aim of this paper is to study convergence of Bousfield-Kan completions with respect to the 1-excisive approximation of the identity functor and exotic convergence of the Taylor tower of the identity functor, for algebras over operads in spectra centered away from the null object. In Goodwillie's homotopy functor calculus, being centered away from the null object amounts to doing homotopy theory and functor calculus in the retractive setting.
\end{abstract}

\maketitle

\section{Introduction}

Let $\capR$ be a $(-1)$-connected commutative ring spectrum (i.e., a commutative monoid object in the category $(\Spectra,\tensor_S,S)$ of symmetric spectra \cite{Hovey_Shipley_Smith, Schwede_book_project}). We denote by $(\ModR,\wedge,\capR)$ the closed symmetric monoidal category of $\capR$-modules and assume that $\capO$ is an operad in $\capR$-modules whose terms are $(-1)$-connected and satisfies $\capO[0]=*$ (i.e., such $\capO$-algebras are non-unital).  We denote by $\AlgO$ the category of $\capO$-algebras in $\capR$-modules; it is pointed by the null object $*$. Let $Y$ be an $\capO$-algebra and denote by $\AlgOY$ the category of retractive $\capO$-algebras over $Y$; in other words, $\AlgOY$ is the factorization category of the identity map on $Y$. The objects in $\AlgOY$ are factorizations $Y\rightarrow X\rightarrow Y$ of the identity map $\function{\id_Y}{Y}{Y}$ in $\AlgO$, and morphisms are the commutative diagrams of the form
\begin{align*}
\xymatrix{
  Y\ar[r]\ar@{=}[d] & X\ar[r]\ar[d] & Y\ar@{=}[d]\\
  Y\ar[r] & X'\ar[r] & Y
}
\end{align*}
in $\AlgO$; it is pointed by the factorization $Y=Y=Y$ of $\id_Y$. 

To keep this paper appropriately concise, we freely use notation from \cite{Ching_Harper}. To understand how to do homotopy theory in $\AlgOY$, see \cite[2.1]{Bauer_Johnson_McCarthy}, \cite[4.9]{Harper_Zhang}; see also \cite[3.10]{Dwyer_Spalinski}, \cite{Schwede_cotangent} and note that $\AlgOY\Iso\id_Y\downarrow(\AlgO\downarrow Y)$ (the category of objects in $(\AlgO\downarrow Y)$ under $\id_Y$ \cite[II.6]{MacLane_categories}); in particular, $\AlgOY$ inherits a simplicial (e.g., \cite{Schwede_cotangent}, \cite[4.2]{Hovey}) cofibrantly generated (e.g., \cite{Hirschhorn_overcategories}) model structure from $\AlgO$  which is equipped with the positive (flat) stable model structure. Our basic assumption is that $Y$ is $(-1)$-connected and both fibrant and cofibrant in $\AlgO$.

If $X$ is a retractive $\capO$-algebra over $Y$, we say that $X$ is \emph{$k$-connected relative to $Y$} (or $k$-connected (rel. $Y$), for short) if the structure map $Y\rightarrow X$ is $k$-connected in $\AlgO$. In Goodwillie's homotopy functor calculus \cite{Goodwillie_calculus_2, Goodwillie_calculus_3, Kuhn_survey}, $\AlgOY$ is the category that naturally arises when studying Taylor towers in $\AlgO$ centered at the object $Y$. The aim of this paper is to extend several results in \cite{Blomquist, Ching_Harper_derived_Koszul_duality, Schonsheck_fibration, Schonsheck_TQ} on Bousfield-Kan completions and Taylor towers for $\capO$-algebras centered at the null object $*$, to analogous results for $\capO$-algebras centered at $Y$. Here are our main results.

In the statement of Theorems \ref{main_theorem_A}, \ref{main_theorem_B}, and \ref{main_theorem_C}, $\Loopt_Y^r$ denotes derived $r$-fold loops (\cite[10.8]{Dwyer_Spalinski}, \cite[I.2]{Quillen}) in $\AlgOY$, $\Suspt_Y^r$ denotes derived $r$-fold suspension (\cite[10.4]{Dwyer_Spalinski}, \cite[I.2]{Quillen}) in $\AlgOY$, $\Loopt_Y^\infty$ denotes the derived version of the 0-th object functor $\Loop_Y^\infty$ (\cite{Hovey_spectra}, Section \ref{sec:proofs_of_main_results}) on Hovey spectra $\SpectraN(\AlgOY)$ on $\AlgOY$, and $\Suspt_Y^\infty$ denotes the derived version of the stabilization functor $\Susp_Y^\infty$ (\cite{Hovey_spectra}, Section \ref{sec:proofs_of_main_results}) on $\AlgOY$. In the special case when $Y=*$, Theorem \ref{main_theorem_A} is proved in \cite{Ching_Harper_derived_Koszul_duality} when $r=\infty$, and subsequently in \cite{Blomquist} (using a different approach, more closely related to ideas in \cite{Blomquist_Harper_integral_chains, Blomquist_Harper, Dundas, Dundas_Goodwillie_McCarthy}) for $1\leq r\leq\infty$. We observe that these results generalize to the context of $\capO$-algebras centered at $Y$. 

\begin{thm}
\label{main_theorem_A}
Assume that $\capR,\capO,Y$ are $(-1)$-connected and $\capO[0]=*$. Let $X$ be a retractive $\capO$-algebra over $Y$. If $X$ is $0$-connected (rel. $Y$),  then the Bousfield-Kan completion maps 
\begin{align*}
  X\wequiv X^\wedge_{\Loopt_Y^r\Suspt_Y^r},\quad\quad (1\leq r\leq \infty)
\end{align*}
are weak equivalences. In particular, since $\Loopt_Y^\infty\Suspt_Y^\infty\wequiv P_1^Y(\id)$, where $P_1^Y(\id)$ denotes the 1-excisive approximation to the identity functor on $\capO$-algebras centered at $Y$, the Bousfield-Kan completion of $X$ with respect to $P_1^Y(\id)$ recovers $X$, up to weak equivalence.
\end{thm}

In the special case when $Y=*$, Theorem \ref{main_theorem_B} is proved in \cite{Schonsheck_fibration} for $r=\infty$ (using an approach closely related to ideas in \cite{Bousfield_Kan}) and Theorem \ref{main_theorem_C} is proved in \cite{Schonsheck_TQ} (using ideas closely related to \cite{Arone_Kankaanrinta}). We observe that these results generalize to the context of $\capO$-algebras centered at $Y$.

\begin{thm}
\label{main_theorem_B}
Assume that $\capR,\capO,Y$ are $(-1)$-connected and $\capO[0]=*$. Consider any fibration sequence of the form
\begin{align*}
  F\rightarrow E\rightarrow B
\end{align*}
in retractive $\capO$-algebras over $Y$. If $E,B$ are 0-connected (rel. $Y)$, then the Bousfield-Kan completion maps
\begin{align}
  \label{eq:bousfield_kan_completion_maps_theorem_B}
  F\wequiv F^\wedge_{\Loopt_Y^r\Suspt_Y^r},\quad\quad (1\leq r\leq \infty)
\end{align}
are weak equivalences. More generally, consider any $\infty$-cartesian 2-cube of the form
\begin{align*}
\xymatrix{
  F\ar[r]\ar[d] & E\ar[d]\\
  Z\ar[r] & B
}
\end{align*}
in retractive $\capO$-algebras over $Y$. If $E,B,Z$ are 0-connected (rel. $Y$), then the Bousfield-Kan completion maps \eqref{eq:bousfield_kan_completion_maps_theorem_B} are weak equivalences.
\end{thm}

\begin{thm}
\label{main_theorem_C}
Assume that $\capR,\capO,Y$ are $(-1)$-connected and $\capO[0]=*$. Let $X$ be a retractive $\capO$-algebra over $Y$. If $X$ is $(-1)$-connected (rel. $Y$) and $P_1^Y(\id)X$ is 0-connected (rel. $Y$), then the homotopy limit of the Taylor tower of the identity functor on $\capO$-algebras centered at $Y$ (and evaluated at $X$) is weakly equivalent
\begin{align*}
  P_\infty^Y(\id)X\wequiv X^\wedge_{\Loopt_Y^\infty\Suspt_Y^\infty}
\end{align*}
to the Bousfield-Kan completion of $X$ with respect to $P_1^Y(\id)$.
\end{thm}

\begin{rem}
\label{rem:about_stuff_useful}
We know from \cite{Ching_Harper} that the identity functor is 0-analytic, but not $(-1)$-analytic, hence it follows (\cite{Goodwillie_calculus_2, Goodwillie_calculus_3}) that any 0-connected (rel $Y$) $X\in\AlgOY$ satisfies $X\wequiv P_\infty^Y(\id)X$; we also know by Theorem \ref{main_theorem_A} that such an $X$ satisfies $X\wequiv X^\wedge_{\Loopt_Y^\infty\Suspt_Y^\infty}$. On the other hand, when $X$ is $(-1)$-connected (rel. $Y$) but its stabilization $P_1^Y(\id)X$ is 0-connected (rel. $Y$), then Theorem \ref{main_theorem_C} provides an interesting convergence result for the Taylor tower of the identity functor. It is worth pointing out that $*_Y\rightarrow\Loopt_Y^\infty\Suspt_Y^\infty X$ is 0-connected in $\AlgOY$ if and  only if $*_Y\rightarrow\Suspt_Y^\infty X$ is 0-connected in $\SpectraN(\AlgOY)$; e.g., by the (derived) triangle identities.
\end{rem}

Our results are enabled by the homotopical resolutions studied in \cite{Blumberg_Riehl}, together with Hovey spectra \cite{Hovey_spectra} of retractive $\capO$-algebras over $Y$ by work on Bousfield localization in model structures that are not left proper; see, for instance, \cite{Goerss_Hopkins_moduli_problems} and \cite{Harper_Zhang}, and the subsequent work of \cite{Batanin_White} and \cite{Carmona}.

\section{Homotopical estimates}

The purpose of this section is to prove Theorem \ref{thm:main_theorem_homotopical_estimates} which provides the detailed homotopical estimates underlying each of our main results. Since making these types of estimates may not be so familiar, we ease the reader into the various higher Blakers-Massey (and its dual) arguments, before building up to a proof of Theorem \ref{thm:main_theorem_homotopical_estimates}. Readers who would like to first see proofs of the main results are invited to look ahead to Section \ref{sec:proofs_of_main_results}.

Denote by $*_Y:=Y$ the null object (or point) in the pointed category $\AlgOY$. It will sometimes be conceptually and notationally convenient to denote by $*_Y'\wequiv *_Y$ an appropriately fattened up version of the point $*_Y$ in $\AlgOY$; this will not cause any confusion. Let's explore the behavior of derived suspension $\Suspt_Y$ (\cite[10.4]{Dwyer_Spalinski}, \cite[I.2]{Quillen}) on objects in $\AlgOY$.

\begin{prop}
Let $k\geq -1$. If $X\in\AlgOY$ is $k$-connected (rel. $Y$), then $\Suspt_Y X$ is $(k+1)$-connected (rel. $Y$).
\end{prop}

\begin{proof}
It suffices to assume that $X$ is cofibrant in $\AlgOY$. Consider a pushout cofibration 2-cube of the form
\begin{align}
\label{eq:suspension_increases_connectivity_by_one}
\xymatrix{
  X\ar[r]^-{(*)}\ar[d]_-{(*)} & {*}_Y'\ar[d]^-{(\#)}\\
  {*}_Y'\ar[r]^-{(\#)} & \Suspt_Y X
}
\end{align}
in $\AlgOY$. By assumption, we know that $*_Y\rightarrow X$ is $k$-connected, hence the maps $(*)$ are $(k+1)$-connected. Since the 2-cube is $\infty$-cocartesian by construction, it follows that the maps $(\#)$ are $(k+1)$-connected; see, for instance, \cite[1.4]{Ching_Harper}.
\end{proof}

Let's explore the behavior of $\Suspt_Y$ on maps in $\AlgOY$; this will provide a useful warmup in using cocartesian-ness estimates.

\begin{prop}
Let $n\geq -1$. Consider any $n$-connected map $A\rightarrow B$ in $\AlgOY$. Assume that $A,B$ are $(-1)$-connected (rel. $Y$). Then the induced map $\Suspt_Y A\rightarrow\Suspt_Y B$ is $(n+1)$-connected.
\end{prop}

\begin{proof}
One way to verify this is to make systematic use of cocartesian-ness estimates. It suffices to assume that $A,B$ are cofibrant in $\AlgOY$. Consider the induced 3-cube of the form
\begin{align*}
\xymatrix@!0{
A\ar[rr]\ar[dd]\ar[dr] &&
{*_Y'}\ar[dr]\ar'[d][dd]\\
&B\ar[rr]\ar[dd] &&
{*_Y'}\ar[dd]\\
{*_Y'}\ar[dr]\ar'[r][rr] &&
\Suspt_Y A\ar[dr]\\
&{*_Y'}\ar[rr] &&
\Suspt_Y B
}
\end{align*}
in $\AlgOY$ where the front and back 2-faces are pushout cofibration 2-cubes (and hence $\infty$-cocartesian). It follows from \cite[3.8]{Ching_Harper} (e.g., \cite[1.7]{Goodwillie_calculus_2}) that the 3-cube is $\infty$-cocartesian. Consider the left-hand 2-face of the form
\begin{align*}
\xymatrix{
  A\ar[r]\ar[d] & B\ar[d]\\
  {*_Y'}\ar[r] & {*_Y'}
}
\end{align*}
Since the upper map is $n$-connected by assumption and the bottom map is $\infty$-connected (i.e., a weak equivalence), it follows from \cite[3.8]{Ching_Harper} (e.g., \cite[1.7]{Goodwillie_calculus_2}) that this 2-face is $(n+1)$-cocartesian. Putting it all together, the 3-cube is $\infty$-cocartesian and the left-hand 2-face is $(n+1)$-cocartesian, hence the right-hand 2-face of the form
\begin{align*}
\xymatrix{
  {*_Y'}\ar[r]\ar[d] & {*_Y'}\ar[d]\\
  \Suspt_Y A\ar[r] & \Suspt_Y B
}
\end{align*}
is $(n+1)$-cocartesian; since the upper map is $\infty$-connected, it follows from \cite[3.8]{Ching_Harper} (e.g., \cite[1.7]{Goodwillie_calculus_2}) that the bottom map is $(n+1)$-connected.
\end{proof}

Let's explore the behavior of derived loops $\Loopt_Y$ (\cite[10.8]{Dwyer_Spalinski}, \cite[I.2]{Quillen}) on objects in $\AlgOY$.

\begin{prop}
Let $k\in\ZZ$. If $X\in\AlgOY$ is $k$-connected (rel. $Y$), then $\Loopt_Y X$ is $(k-1)$-connected (rel. $Y$).
\end{prop}

\begin{proof}
It suffices to assume that $X$ is fibrant in $\AlgOY$. Consider a pullback fibration 2-cube of the form
\begin{align*}
\xymatrix{
  \Loopt_Y X\ar[r]^-{(\#)}\ar[d]_-{(\#)} & {*}_Y'\ar[d]^-{(*)}\\
  {*}_Y'\ar[r]^-{(*)} & X
}
\end{align*}
in $\AlgOY$. By assumption, we know that $*_Y\rightarrow X$ is $k$-connected, hence the maps $(*)$ are $k$-connected. Since the 2-cube is $\infty$-cartesian by construction, it follows that the maps $(\#)$ are $k$-connected.  Hence the map $*_Y\rightarrow\Loopt_Y X$ is $(k-1)$-connected.
\end{proof}

Let's explore the behavior of $\Loopt_Y$ on maps in $\AlgOY$; this will provide a useful warmup in using cartesian-ness estimates.

\begin{prop}
Let $n\in\ZZ$. Consider any $n$-connected map $A\rightarrow B$ in $\AlgOY$. Then the induced map $\Loopt_Y A\rightarrow\Loopt_Y B$ is $(n-1)$-connected.
\end{prop}

\begin{proof}
One way to verify this is to make systematic use of cartesian-ness estimates. It suffices to assume that $A,B$ are fibrant in $\AlgOY$. Consider the induced 3-cube of the form
\begin{align*}
\xymatrix@!0{
\Loopt_Y A\ar[rr]\ar[dd]\ar[dr] &&
{*_Y'}\ar[dr]\ar'[d][dd]\\
&\Loopt_Y B\ar[rr]\ar[dd] &&
{*_Y'}\ar[dd]\\
{*_Y'}\ar[dr]\ar'[r][rr] &&
A\ar[dr]\\
&{*_Y'}\ar[rr] &&
B
}
\end{align*}
in $\AlgOY$ where the front and back 2-faces are pullback fibration 2-cubes (and hence $\infty$-cartesian). It follows from several applications of \cite[3.8]{Ching_Harper} (e.g., \cite[1.6]{Goodwillie_calculus_2}) that the left-hand 2-face is $(n-1)$-cartesian, and hence by another application of \cite[3.8]{Ching_Harper} (e.g., \cite[1.6]{Goodwillie_calculus_2}), we know that the map $\Loopt_Y A\rightarrow\Loopt_Y B$ is $(n-1)$-connected.
\end{proof}

The following, which appears in \cite[6.7]{Beardsley_Lawson}, can be understood as a consequence of Blakers-Massey \cite[1.5]{Ching_Harper} for $\AlgO$.

\begin{prop}
\label{prop:special_case_of_zero_cube}
Let $k\geq -1$. If $X\in\AlgOY$ is $k$-connected (rel. $Y$), then the map of the form $X\rightarrow\Loopt_Y\Suspt_Y X$ in $\AlgOY$ is $(2k+2)$-connected.
\end{prop}

\begin{proof}
It suffices to assume that $X$ is cofibrant in $\AlgOY$. Consider a pushout cofibration 2-cube of the form \eqref{eq:suspension_increases_connectivity_by_one}. By assumption, we know that the maps $(*)$ are $(k+1)$-connected. Since the 2-cube is $\infty$-cocartesian by construction, it follows that the maps $(\#)$ are $(k+1)$-connected; see, for instance, \cite[1.4]{Ching_Harper}. To estimate the connectivity of the map $X\rightarrow\Loopt_Y\Suspt_Y X$, it suffices to estimate the cartesian-ness of the 2-cube. By Blakers-Massey \cite[1.5]{Ching_Harper} for $\AlgO$, we know that diagram \eqref{eq:suspension_increases_connectivity_by_one} is $l$-cartesian where $l$ is the minimum of
\begin{align*}
  -2 + l_{\{1,2\}} + 1 &= -1 + \infty \\
  -2 + l_{\{1\}} + 1 + l_{\{2\}} + 1 &= (k+1) + (k+1) = 2k+2
\end{align*}
Hence $l=2k+2$, our 2-cube is $(2k+2)$-cartesian, and hence the map $X\rightarrow\Loopt_Y\Suspt_Y X$ is $(2k+2)$-connected.
\end{proof}

The next step, on our way to proving Theorem \ref{thm:main_theorem_homotopical_estimates}, is to observe some low-dimensional patterns that naturally arise in the homotopical estimates; this will be useful. Let $\capX$ be a 0-cube in $\AlgOY$ of the form $\capX_\emptyset$ and consider the maps
\begin{align*}
  *_Y\xrightarrow{(*)} \capX_\emptyset\xrightarrow{(\#)} *_Y
\end{align*} 
in $\AlgO$. If $\capX_\emptyset$ is 0-connected (rel. $Y$), then $(*)$ is 0-connected and $(\#)$ is 1-connected; hence the 0-cube $\capX$ is 0-cocartesian and 1-cartesian in $\AlgOY$. These three conditions are tautologically the same: $\capX_\emptyset$ is 0-connected (rel. $Y$) if and only if the 0-cube $\capX$ is 1-cartesian in $\AlgOY$ if and only if the 0-cube $\capX$ is 0-cocartesian in $\AlgOY$. By Proposition \ref{prop:special_case_of_zero_cube}, the 1-cube of the form $\capX\rightarrow \Loopt_Y\Suspt_Y\capX$ in $\AlgOY$, which has the form
\begin{align}
\label{eq:1_cube_for_X_0_connected}
  \capX_\emptyset\rightarrow \Loopt_Y\Suspt_Y\capX_\emptyset
\end{align}
is 2-connected; furthermore, the 0-cubes are 0-connected (rel. $Y$). In other words, the 1-cube \eqref{eq:1_cube_for_X_0_connected} satisfies (i) the 0-subcubes are 1-cartesian and the 1-subcubes are 2-cartesian in $\AlgOY$ (i.e., the 1-cube \eqref{eq:1_cube_for_X_0_connected} is $(\id+1)$-cartesian in $\AlgOY$) and (ii) the 0-subcubes are 0-cocartesian and the 1-subcubes are 2-cocartesian (i.e., the 1-cube \eqref{eq:1_cube_for_X_0_connected} is $(2\cdot\id)$-cocartesian in $\AlgOY$). These conditions are tautologically the same: a 1-cube is $(\id+1)$-cartesian in $\AlgOY$ if and only if it is $(2\cdot\id)$-cocartesian in $\AlgOY$. 

With these observations, we can restate the $k=0$ case of Proposition \ref{prop:special_case_of_zero_cube} as follows: If the 0-cube $\capX$ is $(\id+1)$-cartesian in $\AlgOY$, then so is the 1-cube of the form $\capX\rightarrow\Loopt_Y\Suspt_Y\capX$.

\begin{rem}
Equivalently, we could have restated the $k=0$ case of Proposition \ref{prop:special_case_of_zero_cube} as follows: If the 0-cube $\capX$ is $(2\cdot\id)$-cocartesian in $\AlgOY$, then so is the 1-cube of the form $\capX\rightarrow\Loopt_Y\Suspt_Y\capX$.
\end{rem}

It turns out that this relationship persists for all higher dimensional $n$-cubes $\capX$ in $\AlgOY$ $(n\geq 0)$, beyond the $n=0$ case of Proposition \ref{prop:special_case_of_zero_cube} above.

\begin{thm}
\label{thm:main_theorem_for_special_case_of_0_connected}
Let $W$ be a finite set and $\capX$ a $W$-cube in $\AlgOY$. Let $n=|W|$. If the $n$-cube $\capX$ is $(\id+1)$-cartesian in $\AlgOY$, then so is the $(n+1)$-cube of the form $\capX\rightarrow\Loopt_Y\Suspt_Y\capX$.
\end{thm}

\begin{rem}
The cartesian-ness estimates in Theorem \ref{thm:main_theorem_for_special_case_of_0_connected} are strong enough to show that if $X\in\AlgOY$ is 0-connected (rel. $Y$), then (i) $X$ is weakly equivalent to its Bousfield-Kan completion with respect to $\Loopt_Y\Suspt_Y$ (i.e., $X\wequiv X^\wedge_{\Loopt_Y\Suspt_Y}$) and (ii) the associated homotopy spectral sequence converges strongly.
\end{rem}

The duality relationship observed above---between a specific uniform cartesian-ness of 1-cubes in $\AlgOY$ and a specific uniform cocartesian-ness of 1-cubes in $\AlgOY$--- persists for all higher dimensional $n$-cubes $\capX$ in $\AlgOY$ $(n\geq 0)$.
 
\begin{prop}
\label{prop:duality_special_case_of_0_connected}
Let $W$ be a finite set and $\capX$ a $W$-cube in $\AlgOY$. Let $n=|W|$. The $n$-cube $\capX$ is $(\id+1)$-cartesian in $\AlgOY$ if and only if it is $(2\cdot\id)$-cocartesian in $\AlgOY$.
\end{prop}

\begin{rem}
In particular, we could have restated the last statement of Theorem \ref{thm:main_theorem_for_special_case_of_0_connected} as follows: If the $n$-cube $\capX$ is $(2\cdot\id)$-cocartesian in $\AlgOY$, then so is the $(n+1)$-cube of the form $\capX\rightarrow\Loopt_Y\Suspt_Y\capX$.
\end{rem}

Let's first work out the proof of Proposition \ref{prop:duality_special_case_of_0_connected}, before proceeding with a proof of Theorem \ref{thm:main_theorem_for_special_case_of_0_connected}; this will provide a useful warmup in making certain cartesian-ness and cocartesian-ness estimates.

\begin{proof}[Proof of Proposition \ref{prop:duality_special_case_of_0_connected}]
Here is the basic idea. As noted above, this is tautologically true for the case of 0-cubes and 1-cubes. Consider the case of 2-cubes. Suppose $\capX$ is a $\{1,2\}$-cube in $\AlgOY$ of the form
\begin{align}
\label{eq:2_cube_X_in_AlgOY}
\xymatrix{
  \capX_\emptyset\ar[r]\ar[d] & \capX_{\{1\}}\ar[d]\\
  \capX_{\{2\}}\ar[r] & \capX_{\{1,2\}}
}
\end{align}
Assume that $\capX$ is $(\id+1)$-cartesian in $\AlgOY$; this means that: the 0-subcubes are 1-cartesian, the 1-subcubes are 2-cartesian, and the 2-subcubes are 3-cartesian. Let's verify that $\capX$ is $(2\cdot\id)$-cocartesian in $\AlgOY$. The first step is to estimate the cocartesian-ness of $\capX$ itself (the only 2-subcube of $\capX$). By higher dual Blakers-Massey \cite[1.11]{Ching_Harper} for $\AlgO$, we know that $\capX$ is $k$-cocartesian where $k$ is the minimum of
\begin{align*}
  k_{\{1,2\}}+2-1 &= 3 + 1\\
  2+k_{\{1\}}+k_{\{2\}} &= 2 + 2 + 2
\end{align*}
Hence $k=4$ and we have calculated that the 2-cube $\capX$ is 4-cocartesian. By the earlier cases for $n=0,1$, together with \cite[3.9]{Ching_Harper} (e.g., \cite[1.8]{Goodwillie_calculus_2}), we have verified that: the 0-subcubes are 0-cocartesian, the 1-subcubes are 2-cocartesian, and the 2-subcubes are 4-cocartesian. Hence $\capX$ is $(2\cdot\id)$-cocartesian in $\AlgOY$. Conversely, assume that $\capX$ is $(2\cdot\id)$-cocartesian in $\AlgOY$. Let's verify that $\capX$ is $(\id+1)$-cartesian in $\AlgOY$. The first step is to estimate the cartesian-ness of $\capX$ itself (the only 2-subcube of $\capX$). By higher Blakers-Massey \cite[1.7]{Ching_Harper} for $\AlgO$, we know that $\capX$ is $k$-cartesian where $k$ is the minimum of
\begin{align*}
  -2+k_{\{1,2\}}+1 &= 4-1\\
  -2+k_{\{1\}}+1+k_{\{2\}}+1 &= 2+2
\end{align*}
Hence $k=3$ and we have calculated that the 2-cube $\capX$ is 3-cartesian. By the earlier cases for $n=0,1$, together with \cite[3.9]{Ching_Harper} (e.g., \cite[1.8]{Goodwillie_calculus_2}), we have verified that: the 0-subcubes are 1-cartesian, the 1-subcubes are 2-cartesian, and the 2-subcubes are 3-cartesian. Hence $\capX$ is $(\id+1)$-cartesian in $\AlgOY$.

Consider the case of 3-cubes. Suppose $\capX$ is a $\{1,2,3\}$-cube in $\AlgOY$. Assume that $\capX$ is $(\id+1)$-cartesian in $\AlgOY$; this means that: the 0-subcubes are 1-cartesian, the 1-subcubes are 2-cartesian, the 2-subcubes are 3-cartesian, and the 3-subcubes are 4-cartesian. Let's verify that $\capX$ is $(2\cdot\id)$-cocartesian in $\AlgOY$. The first step is to estimate the cocartesian-ness of $\capX$ itself (the only 3-subcube of $\capX$). By higher dual Blakers-Massey \cite[1.11]{Ching_Harper} for $\AlgO$, we know that $\capX$ is $k$-cocartesian where $k$ is the minimum of
\begin{align*}
  k_{\{1,2,3\}}+3-1 &= 4+2\\
  3+k_{\{1,2\}} + k_{\{3\}} &= 3+3+2\\
  3+k_{\{1\}}+k_{\{2\}}+k_{\{3\}} &= 3+2+2+2
\end{align*}
Note that by uniformity of the cartesian-ness estimates, allowing other partitions of $\{1,2,3\}$ provides nothing new; hence we have left them out of the above. Hence $k=6$ and we have calculated that the 3-cube $\capX$ is 6-cocartesian. By the earlier cases for $n=0,1,2$, together with \cite[3.9]{Ching_Harper} (e.g., \cite[1.8]{Goodwillie_calculus_2}), we have verified that: the 0-subcubes are 0-cocartesian, the 1-subcubes are 2-cocartesian, the 2-subcubes are 4-cocartesian, and the 3-subcubes are 6-cocartesian. Hence $\capX$ is $(2\cdot\id)$-cocartesian in $\AlgOY$. Conversely, assume that $\capX$ is $(2\cdot\id)$-cocartesian in $\AlgOY$. Let's verify that $\capX$ is $(\id+1)$-cartesian in $\AlgOY$. The first step is to estimate the cartesian-ness of $\capX$ itself (the only 3-subcube of $\capX$). By higher Blakers-Massey \cite[1.7]{Ching_Harper} for $\AlgO$, we know that $\capX$ is $k$-cartesian where $k$ is the minimum of
\begin{align*}
  -3+k_{\{1,2,3\}}+1 &= -2 + 6\\
  -3+k_{\{1,2\}}+1+k_{\{3\}}+1 &= -1+4+2\\
  -3+k_{\{1\}}+1+k_{\{2\}}+1+k_{\{3\}}+1 &= 2+2+2
\end{align*}
Note that by uniformity of the cocartesian-ness estimates, allowing other partitions of $\{1,2,3\}$ provides nothing new; hence we have left them out of the above. Hence $k=4$ and we have calculated that the 3-cube $\capX$ is 4-cartesian. By the earlier cases for $n=0,1,2$, together with \cite[3.9]{Ching_Harper} (e.g., \cite[1.8]{Goodwillie_calculus_2}), we have verified that: the 0-subcubes are 1-cartesian, the 1-subcubes are 2-cartesian, the 2-subcubes are 3-cartesian, and the 3-subcubes are 4-cartesian. Hence $\capX$ is $(\id+1)$-cartesian in $\AlgOY$. And so forth.
\end{proof}

\begin{proof}[Proof of Theorem \ref{thm:main_theorem_for_special_case_of_0_connected}]
Here is the basic idea. The case of 0-cubes is given in Proposition \ref{prop:special_case_of_zero_cube} above. Consider the case of 1-cubes. Suppose $\capX$ is a $\{1\}$-cube in $\AlgOY$ of the form $\capX_\emptyset\rightarrow\capX_{\{1\}}$. 
Assume that $\capX$ is $(\id+1)$-cartesian in $\AlgOY$; this means that: the 0-subcubes are 1-cartesian and the 1-subcubes are 2-cartesian. Let's verify that the 2-cube of the form $\capX\rightarrow\Loopt_Y\Suspt_Y\capX$ is $(\id+1)$-cartesian in $\AlgOY$. It suffices to assume that $\capX$ is a cofibration 1-cube. Let $C$ be the homotopy cofiber of $\capX_\emptyset\rightarrow\capX_{\{1\}}$ in $\AlgOY$ and consider the associated $\infty$-cocartesian 2-cube of the form
\begin{align}
\label{eq:the_2_cube_X_to_C_special_case_0_connected}
\xymatrix{
  \capX:\ar[d] & \capX_\emptyset\ar[r]\ar[d] & \capX_{\{1\}}\ar[d]\\
  \capC: & {*_Y}\ar[r] & C
}
\end{align}
in $\AlgOY$, where $\capC$ is the indicated 1-face on the bottom. By Proposition \ref{prop:duality_special_case_of_0_connected}, we know that $\capX$ is $(2\cdot\id)$-cocartesian; in particular, $\capX$ is 2-cocartesian and hence $C$ is 2-connected (rel. $Y$). Putting it all together, it follows that the vertical maps in diagram \eqref{eq:the_2_cube_X_to_C_special_case_0_connected} are 1-connected and the horizontal maps are 2-connected. Here is our strategy: consider the commutative diagram of 2-cubes in $\AlgOY$ of the form
\begin{align}
\label{eq:the_strategy_picture_special_case_0_connected}
\xymatrix{
  \capX\ar[r]^-{(a)}\ar[d]_-{(*)} & \capC\ar[d]^-{(c)}\\
  \Loopt_Y\Suspt_Y\capX\ar[r]^-{(b)} & \Loopt_Y\Suspt_Y\capC
}
\end{align}
Instead of attempting to estimate the cartesian-ness of the 2-cube $(*)$ directly, which seems difficult, we will take an indirect attack and first estimate the cartesian-ness of the 2-cubes $(a),(b),(c)$; then we will use \cite[3.9]{Ching_Harper} (e.g., \cite[1.8]{Goodwillie_calculus_2}) to deduce an estimate for $(*)$. Consider the 2-cube $(a)$. By higher Blakers-Massey \cite[1.7]{Ching_Harper} for $\AlgO$, we know that $(a)$ is $k$-cartesian where $k$ is the minimum of
\begin{align*}
  -2+k_{\{1,2\}}+1 &= -1+\infty\\
  -2+k_{\{1\}}+1+k_{\{2\}}+1 &= 2+1
\end{align*}
Hence $k=3$ and we have calculated that the 2-cube $(a)$ is 3-cartesian. Consider the 2-cube $(b)$. The 2-cube $\Suspt_Y\capX\rightarrow\Suspt_Y\capC$ is $\infty$-cocartesian and has the form
\begin{align}
\label{eq:suspension_of_the_2_cube_X_to_C_special_case_0_connected}
\xymatrix{
  \Suspt_Y\capX:\ar[d] & \Suspt_Y\capX_\emptyset\ar[r]\ar[d] & \Suspt_Y\capX_{\{1\}}\ar[d]\\
  \Suspt_Y\capC: & \Suspt_Y{*_Y}\ar[r] & \Suspt_YC
}
\end{align}
in $\AlgOY$, where $\Suspt_Y{*_Y}\wequiv{*_Y}$. It follows that the vertical maps in diagram \eqref{eq:suspension_of_the_2_cube_X_to_C_special_case_0_connected} are 2-connected and the horizontal maps are 3-connected. By higher Blakers-Massey \cite[1.7]{Ching_Harper} for $\AlgO$, we know that $\Suspt_Y\capX\rightarrow\Suspt_Y\capC$ is $k$-cartesian where $k$ is the minimum of
\begin{align*}
  -2+k_{\{1,2\}}+1 &= -1+\infty\\
  -2+k_{\{1\}}+1+k_{\{2\}}+1 &= 3+2
\end{align*}
Hence $k=5$ and we have calculated that the 2-cube $\Suspt_Y\capX\rightarrow\Suspt_Y\capC$ is 5-cartesian; therefore the 2-cube $(b)$ is 4-cartesian. Consider the 2-cube (c). We know that $C$ is 2-connected (rel. $Y$) from above, hence by Proposition \ref{prop:special_case_of_zero_cube} the map $C\rightarrow\Loopt_Y\Suspt_Y C$ is $(2\cdot 2+2)$-connected. This calculation will produce a cartesian-ness estimate for the 2-cube $(c)$. Here is why: the 2-cube (c) has the form
\begin{align}
\label{eq:the_2_cube_labelled_c}
\xymatrix{
  {*_Y}\ar[r]\ar[d] & C\ar[d]^-{(\#)}\\
  \Loopt_Y\Suspt_Y{*_Y}\ar[r] & \Loopt_Y\Suspt_Y C
}
\end{align}
in $\AlgOY$, where $\Loopt_Y\Suspt_Y{*_Y}\wequiv {*_Y}$. Taking homotopy fibers horizontally produces the map $\Loopt_Y C\rightarrow\Loopt_Y\Loopt_Y\Suspt_Y C$; this map is $\Loopt_Y$ of the right-hand vertical map. We know $(\#)$ is 6-connected from above; hence the 2-cube $(c)$ is 5-cartesian. Putting it all together, it follows from diagram \eqref{eq:the_strategy_picture_special_case_0_connected} and \cite[3.9]{Ching_Harper} (e.g., \cite[1.8]{Goodwillie_calculus_2}), together with our cartesian-ness estimates for $(a),(b),(c)$,  that the 2-cube $(*)$ of the form
\begin{align}
\label{eq:the_2_cube_labelled_star}
\xymatrix{
  \capX:\ar[d]_-{(*)} & \capX_\emptyset\ar[r]\ar[d] & \capX_{\{1\}}\ar[d]\\
  \Loopt_Y\Suspt_Y\capX: & \Loopt_Y\Suspt_Y\capX_\emptyset\ar[r] & 
  \Loopt_Y\Suspt_Y\capX_{\{1\}}
}
\end{align}
in $\AlgOY$, is 3-cartesian. Finally, note that since the top horizontal map is 2-connected, it follows that the bottom horizontal map is 2-connected; i.e., the bottom horizontal 1-face is 2-cartesian. Hence we have verified that the 2-cube $\capX\rightarrow\Loopt_Y\Suspt_Y\capX$ satisfies: the 0-subcubes are 1-cartesian, the 1-subcubes are 2-cartesian, and the 2-subcubes are 3-cartesian. Therefore, the 2-cube of the form $\capX\rightarrow\Loopt_Y\Suspt_Y\capX$ is $(\id+1)$-cartesian in $\AlgOY$.

Consider the case of 2-cubes. Suppose $\capX$ is a $\{1,2\}$-cube in $\AlgOY$ of the form \eqref{eq:2_cube_X_in_AlgOY}. Assume that $\capX$ is $(\id+1)$-cartesian in $\AlgOY$; this means that: the 0-subcubes are 1-cartesian, the 1-subcubes are 2-cartesian, and the 2-subcubes are 3-cartesian. Let's verify that the 3-cube of the form $\capX\rightarrow\Loopt_Y\Suspt_Y\capX$ is $(\id+1)$-cartesian in $\AlgOY$. It suffices to assume that $\capX$ is a cofibration 2-cube. Let $C$ be the iterated homotopy cofiber of $\capX$ in $\AlgOY$ and consider the associated $\infty$-cocartesian 3-cube of the form
\begin{align}
\label{eq:the_3_cube_X_to_C_special_case_0_connected}
\xymatrix@!0{
\capX:\ar[dd]&&\capX_\emptyset\ar[dd]\ar[rr]\ar[dr] &&
\capX_{\{1\}}\ar'[d][dd]\ar[dr]\\
&&&\capX_{\{2\}}\ar[dd]\ar[rr] &&
\capX_{\{1,2\}}\ar[dd]\\
\capC:&&{*_Y}\ar'[r][rr]\ar[dr] &&
{*_Y}\ar[dr]\\
&&&{*_Y}\ar[rr] &&
C
}
\end{align}
in $\AlgOY$, where $\capC$ is the indicated 2-face on the bottom. By Proposition \ref{prop:duality_special_case_of_0_connected}, we know that $\capX$ is $(2\cdot\id)$-cocartesian; in particular, $\capX$ is 4-cocartesian and hence $C$ is 4-connected (rel. $Y$). Putting it all together, it follows that the vertical maps in diagram \eqref{eq:the_3_cube_X_to_C_special_case_0_connected} are 1-connected, the horizontal maps in the top 2-face (i.e., the 1-faces of $\capX$) are 2-connected, and the top 2-face is 4-cocartesian. We would like to estimate the cocartesian-ness of the back 2-face of the form
\begin{align}
\label{eq:back_two_face_of_the_form}
\xymatrix{
  \capX_\emptyset\ar[r]\ar[d] & \capX_{\{1\}}\ar[d]\\
  {*_Y}\ar[r] & {*_Y}
}
\end{align}
The upper horizontal map is 2-connected and the lower horizontal map is $\infty$-connected (i.e., a weak equivalence), hence by \cite[3.8]{Ching_Harper} the back 2-face in \eqref{eq:the_3_cube_X_to_C_special_case_0_connected} is 3-cocartesian in $\AlgOY$. Similarly, the left-hand 2-face in \eqref{eq:the_3_cube_X_to_C_special_case_0_connected} is 3-cocartesian in $\AlgOY$. Here is our strategy: consider the commutative diagram of 3-cubes in $\AlgOY$ of the form \eqref{eq:the_strategy_picture_special_case_0_connected}. Instead of attempting to estimate the cartesian-ness of the 3-cube $(*)$ directly, which seems difficult, we will take an indirect attack and first estimate the cartesian-ness of the 3-cubes $(a),(b),(c)$; then we will use \cite[3.9]{Ching_Harper} (e.g., \cite[1.8]{Goodwillie_calculus_2}) to deduce an estimate for $(*)$. Consider the 3-cube $(a)$. By higher Blakers-Massey \cite[1.7]{Ching_Harper} for $\AlgO$, we know that $(a)$ is $k$-cartesian where $k$ is the minimum of
\begin{align*}
  -3+k_{\{1,2,3\}}+1 &= -2+\infty\\
  -3+k_{\{1,2\}}+1+k_{\{3\}}+1 &= -1+4+1\\
  -3+k_{\{1,3\}}+1+k_{\{2\}}+1 &= -1+3+2\\
  -3+k_{\{2,3\}}+1+k_{\{1\}}+1 &= -1+3+2\\
  -3+k_{\{1\}}+1+k_{\{2\}}+1+k_{\{3\}}+1 &= 2+2+1
\end{align*}
Hence $k=4$ and we have calculated that the 3-cube $(a)$ is 4-cartesian. Consider the 3-cube $(b)$. The 3-cube $\Suspt_Y\capX\rightarrow\Suspt_Y\capC$ is $\infty$-cocartesian in $\AlgOY$, where $\Suspt_Y{*_Y}\wequiv{*_Y}$. By higher Blakers-Massey \cite[1.7]{Ching_Harper} for $\AlgO$, we know that $\Suspt_Y\capX\rightarrow\Suspt_Y\capC$ is $k$-cartesian where $k$ is the minimum of
\begin{align*}
  -3+k_{\{1,2,3\}}+1 &= -2+\infty\\
  -3+k_{\{1,2\}}+1+k_{\{3\}}+1 &= -1+5+2\\
  -3+k_{\{1,3\}}+1+k_{\{2\}}+1 &= -1+4+3\\
  -3+k_{\{2,3\}}+1+k_{\{1\}}+1 &= -1+4+3\\
  -3+k_{\{1\}}+1+k_{\{2\}}+1+k_{\{3\}}+1 &= 3+3+2
\end{align*}
Hence $k=6$ and we have calculated that the 3-cube $\Suspt_Y\capX\rightarrow\Suspt_Y\capC$ is 6-cartesian; therefore the 3-cube $(b)$ is 5-cartesian. Consider the 3-cube (c). We know that $C$ is 4-connected (rel. $Y$) from above, hence by Proposition \ref{prop:special_case_of_zero_cube} the map $C\rightarrow\Loopt_Y\Suspt_Y C$ is $(2\cdot 4+2)$-connected. This calculation will produce a cartesian-ness estimate for the 3-cube $(c)$. Here is why: the 3-cube (c) has the form
\begin{align}
\label{eq:3_cube_c_has_the_form}
\xymatrix@!0{
{*_Y}\ar[dd]\ar[rr]\ar[dr] &&
{*_Y}\ar'[d][dd]\ar[dr]\\
&{*_Y}\ar[dd]\ar[rr] &&
C\ar[dd]^-{(\#)}\\
{*_Y'}\ar'[r][rr]\ar[dr] &&
{*_Y'}\ar[dr]\\
&{*_Y'}\ar[rr] &&
\Loopt_Y\Suspt_Y C
}
\end{align}
in $\AlgOY$, where $\Loopt_Y\Suspt_Y{*_Y}\wequiv {*_Y'}$. Taking homotopy fibers (twice) in $\AlgOY$ produces the map $\Loopt^2_Y C\rightarrow\Loopt^2_Y\Loopt_Y\Suspt_Y C$; this map is $\Loopt^2_Y$ of the right-hand vertical map $(\#)$. We know $(\#)$ is 10-connected from above; hence the 3-cube $(c)$ is 8-cartesian. Putting it all together, it follows from diagram \eqref{eq:the_strategy_picture_special_case_0_connected} and \cite[3.9]{Ching_Harper} (e.g., \cite[1.8]{Goodwillie_calculus_2}), together with our cartesian-ness estimates for $(a),(b),(c)$,  that the 3-cube $(*)$ of the form
\begin{align*}
\xymatrix{
  \capX:\ar[d]_-{(*)} & \capX_\emptyset\ar[r]\ar[d] & \capX_{\{1\}}\ar[d]\\
  \Loopt_Y\Suspt_Y\capX: & \Loopt_Y\Suspt_Y\capX_\emptyset\ar[r] & 
  \Loopt_Y\Suspt_Y\capX_{\{1\}}
}
\end{align*}
in $\AlgOY$, is 4-cartesian. Let's calculate a cartesian-ness estimate for the 2-subcube $\Loopt_Y\Suspt_Y\capX$ of $(*)$. We know that $\capX$ is $(2\cdot\id)$-cocartesian in $\AlgOY$ from above, hence $\Suspt_Y\capX$ is $(2\cdot\id+1)$-cocartesian in $\AlgOY$. By higher Blakers-Massey \cite[1.7]{Ching_Harper} for $\AlgO$, we know that $\Suspt_Y\capX$ is $k$-cartesian where $k$ is the minimum of
\begin{align*}
  -2+k_{\{1,2\}}+1 &= -1+5\\
  -2+k_{\{1\}}+1+k_{\{2\}}+1 &= 3+3
\end{align*}
 Hence $k=4$ and we have calculated that the 2-cube $\Suspt_Y\capX$ is 4-cartesian; therefore the 2-cube $\Loopt_Y\Suspt_Y\capX$ is 3-cartesian. Hence we have verified that the 3-cube $\capX\rightarrow\Loopt_Y\Suspt_Y\capX$ satisfies: the 0-subcubes are 1-cartesian, the 1-subcubes are 2-cartesian, the 2-subcubes are 3-cartesian, and the 3-subcubes are 4-cartesian. Therefore, the 3-cube of the form $\capX\rightarrow\Loopt_Y\Suspt_Y\capX$ is $(\id+1)$-cartesian in $\AlgOY$. And so forth.
\end{proof}

\subsection{The general case $k\geq 0$}

It will be useful to have similar cartesian-ness and cocartesian-ness estimates for $k\geq 0$, beyond the $k=0$ case when $X\in\AlgOY$ is 0-connected (rel. $Y$). For this purpose, let's revisit our above motivating remarks in the case when $k\geq 0$.

Let $k\geq 0$. Let $\capX$ be a 0-cube in $\AlgOY$ of the form $\capX_\emptyset$ and consider the maps
\begin{align*}
  *_Y\xrightarrow{(*)} \capX_\emptyset\xrightarrow{(\#)} *_Y
\end{align*} 
in $\AlgO$. If $\capX_\emptyset$ is $k$-connected (rel. $Y$), then $(*)$ is $k$-connected and $(\#)$ is $(k+1)$-connected; hence the 0-cube $\capX$ is $k$-cocartesian and $(k+1)$-cartesian in $\AlgOY$. These three conditions are tautologically the same: $\capX_\emptyset$ is $k$-connected (rel. $Y$) if and only if the 0-cube $\capX$ is $(k+1)$-cartesian in $\AlgOY$ if and only if the 0-cube $\capX$ is $k$-cocartesian in $\AlgOY$. By Proposition \ref{prop:special_case_of_zero_cube}, the 1-cube of the form $\capX\rightarrow \Loopt_Y\Suspt_Y\capX$ in $\AlgOY$, which has the form
\begin{align}
\label{eq:1_cube_for_X_k_connected}
  \capX_\emptyset\rightarrow \Loopt_Y\Suspt_Y\capX_\emptyset
\end{align}
is $(2k+2)$-connected; furthermore, the 0-cubes are $k$-connected (rel. $Y$). In other words, the 1-cube \eqref{eq:1_cube_for_X_k_connected} satisfies (i) the 0-subcubes are $(k+1)$-cartesian and the 1-subcubes are $(2k+2)$-cartesian in $\AlgOY$ (i.e., the 1-cube \eqref{eq:1_cube_for_X_k_connected} is $(\id+1)(k+1)$-cartesian in $\AlgOY$) and (ii) the 0-subcubes are $k$-cocartesian and the 1-subcubes are $(2k+2)$-cocartesian (i.e., the 1-cube \eqref{eq:1_cube_for_X_k_connected} is $(\id(k+2)+k)$-cocartesian in $\AlgOY$). These conditions are tautologically the same: a 1-cube is $(\id+1)(k+1)$-cartesian in $\AlgOY$ if and only if it is $(\id(k+2)+k)$-cocartesian in $\AlgOY$. 

With these observations, we can restate Proposition \ref{prop:special_case_of_zero_cube} as follows: If the 0-cube $\capX$ is $(\id+1)(k+1)$-cartesian in $\AlgOY$, then so is the 1-cube of the form $\capX\rightarrow\Loopt_Y\Suspt_Y\capX$.

\begin{rem}
Equivalently, we could have restated Proposition \ref{prop:special_case_of_zero_cube} as follows: If the 0-cube $\capX$ is $(\id(k+2)+k)$-cocartesian in $\AlgOY$, then so is the 1-cube of the form $\capX\rightarrow\Loopt_Y\Suspt_Y\capX$.
\end{rem}

It turns out that this relationship persists for all higher dimensional $n$-cubes $\capX$ in $\AlgOY$ $(n\geq 0)$, beyond the $n=0$ case of Proposition \ref{prop:special_case_of_zero_cube} above.

\begin{thm}
\label{thm:main_theorem_for_special_case_of_k_connected}
Let $k\geq 0$. Let $W$ be a finite set and $\capX$ a $W$-cube in $\AlgOY$. Let $n=|W|$. If the $n$-cube $\capX$ is $(\id+1)(k+1)$-cartesian in $\AlgOY$, then so is the $(n+1)$-cube of the form $\capX\rightarrow\Loopt_Y\Suspt_Y\capX$.
\end{thm}

The duality relationship observed above---between a specific uniform cartesian-ness of 1-cubes in $\AlgOY$ and a specific uniform cocartesian-ness of 1-cubes in $\AlgOY$--- persists for all higher dimensional $n$-cubes $\capX$ in $\AlgOY$ $(n\geq 0)$.
 
\begin{prop}
\label{prop:duality_special_case_of_k_connected}
Let $W$ be a finite set and $\capX$ a $W$-cube in $\AlgOY$. Let $n=|W|$. The $n$-cube $\capX$ is $(\id+1)(k+1)$-cartesian in $\AlgOY$ if and only if it is $(\id(k+2)+k)$-cocartesian in $\AlgOY$.
\end{prop}

\begin{rem}
In particular, we could have restated the last statement of Theorem \ref{thm:main_theorem_for_special_case_of_k_connected} as follows: If the $n$-cube $\capX$ is $(\id(k+2)+k)$-cocartesian in $\AlgOY$, then so is the $(n+1)$-cube of the form $\capX\rightarrow\Loopt_Y\Suspt_Y\capX$.
\end{rem}

Let's first work out the proof of Proposition \ref{prop:duality_special_case_of_k_connected}, before proceeding with a proof of Theorem \ref{thm:main_theorem_for_special_case_of_k_connected}; this will provide a useful warmup in making certain cartesian-ness and cocartesian-ness estimates.

\begin{proof}[Proof of Proposition \ref{prop:duality_special_case_of_k_connected}]
These estimates were observed in \cite{Blomquist} for the special case of $Y=*$, and they remain true in the context of $\capO$-algebras centered at $Y$. Here is the basic idea. As noted above, this is tautologically true for the case of 0-cubes and 1-cubes. Consider the case of 2-cubes. Suppose $\capX$ is a $\{1,2\}$-cube in $\AlgOY$ of the form \eqref{eq:2_cube_X_in_AlgOY}. Assume that $\capX$ is $(\id+1)(k+1)$-cartesian in $\AlgOY$; this means that: the 0-subcubes are $(k+1)$-cartesian, the 1-subcubes are $(2k+2)$-cartesian, and the 2-subcubes are $(3k+3)$-cartesian. Let's verify that $\capX$ is $(\id(k+2)+k)$-cocartesian in $\AlgOY$. The first step is to estimate the cocartesian-ness of $\capX$ itself (the only 2-subcube of $\capX$). By higher dual Blakers-Massey \cite[1.11]{Ching_Harper} for $\AlgO$, we know that $\capX$ is $l$-cocartesian where $l$ is the minimum of
\begin{align*}
  l_{\{1,2\}}+2-1 &= 1+(3k+3)\\
  2+l_{\{1\}}+l_{\{2\}} &= 2 + (2k+2)+(2k+2)
\end{align*}
Hence $l=3k+4$ and we have calculated that the 2-cube $\capX$ is $(3k+4)$-cocartesian. By the earlier cases for $n=0,1$, together with \cite[3.9]{Ching_Harper} (e.g., \cite[1.8]{Goodwillie_calculus_2}), we have verified that: the 0-subcubes are $k$-cocartesian, the 1-subcubes are $(2k+2)$-cocartesian, and the 2-subcubes are $(3k+4)$-cocartesian. Hence $\capX$ is $(\id(k+2)+k)$-cocartesian in $\AlgOY$. Conversely, assume that $\capX$ is $(\id(k+2)+k)$-cocartesian in $\AlgOY$. Let's verify that $\capX$ is $(\id+1)(k+1)$-cartesian in $\AlgOY$. The first step is to estimate the cartesian-ness of $\capX$ itself (the only 2-subcube of $\capX$). By higher Blakers-Massey \cite[1.7]{Ching_Harper} for $\AlgO$, we know that $\capX$ is $l$-cartesian where $l$ is the minimum of
\begin{align*}
  -2+l_{\{1,2\}}+1 &= (3k+4)-1\\
  -2+l_{\{1\}}+1+l_{\{2\}}+1 &= (2k+2)+(2k+2)
\end{align*}
Hence $l=3k+3$ and we have calculated that the 2-cube $\capX$ is $(3k+3)$-cartesian. By the earlier cases for $n=0,1$, together with \cite[3.9]{Ching_Harper} (e.g., \cite[1.8]{Goodwillie_calculus_2}), we have verified that: the 0-subcubes are $(k+1)$-cartesian, the 1-subcubes are $(2k+2)$-cartesian, and the 2-subcubes are $(3k+3)$-cartesian. Hence $\capX$ is $(\id+1)(k+1)$-cartesian in $\AlgOY$.

Consider the case of 3-cubes. Suppose $\capX$ is a $\{1,2,3\}$-cube in $\AlgOY$. Assume that $\capX$ is $(\id+1)(k+1)$-cartesian in $\AlgOY$; this means that: the 0-subcubes are $(k+1)$-cartesian, the 1-subcubes are $(2k+2)$-cartesian, the 2-subcubes are $(3k+3)$-cartesian, and the 3-subcubes are $(4k+4)$-cartesian. Let's verify that $\capX$ is $(\id(k+2)+k)$-cocartesian in $\AlgOY$. The first step is to estimate the cocartesian-ness of $\capX$ itself (the only 3-subcube of $\capX$). By higher dual Blakers-Massey \cite[1.11]{Ching_Harper} for $\AlgO$, we know that $\capX$ is $l$-cocartesian where $l$ is the minimum of
\begin{align*}
  l_{\{1,2,3\}}+3-1 &= (4k+4)+2\\
  3+l_{\{1,2\}} + l_{\{3\}} &= 3+(3k+3)+(2k+2)\\
  3+l_{\{1\}}+l_{\{2\}}+l_{\{3\}} &= 3+(2k+2)+(2k+2)+(2k+2)
\end{align*}
Note that by uniformity of the cartesian-ness estimates, allowing other partitions of $\{1,2,3\}$ provides nothing new; hence we have left them out of the above. Hence $l=4k+6$ and we have calculated that the 3-cube $\capX$ is $(4k+6)$-cocartesian. By the earlier cases for $n=0,1,2$, together with \cite[3.9]{Ching_Harper} (e.g., \cite[1.8]{Goodwillie_calculus_2}), we have verified that: the 0-subcubes are $k$-cocartesian, the 1-subcubes are $(2k+2)$-cocartesian, the 2-subcubes are $(3k+4)$-cocartesian, and the 3-subcubes are $(4k+6)$-cocartesian. Hence $\capX$ is $(\id(k+2)+k)$-cocartesian in $\AlgOY$. Conversely, assume that $\capX$ is $(\id(k+2)+k)$-cocartesian in $\AlgOY$. Let's verify that $\capX$ is $(\id+1)(k+1)$-cartesian in $\AlgOY$. The first step is to estimate the cartesian-ness of $\capX$ itself (the only 3-subcube of $\capX$). By higher Blakers-Massey \cite[1.7]{Ching_Harper} for $\AlgO$, we know that $\capX$ is $l$-cartesian where $l$ is the minimum of
\begin{align*}
  -3+l_{\{1,2,3\}}+1 &= -2 + (4k+6)\\
  -3+l_{\{1,2\}}+1+l_{\{3\}}+1 &= -1+(3k+4)+(2k+2)\\
  -3+l_{\{1\}}+1+l_{\{2\}}+1+l_{\{3\}}+1 &= (2k+2)+(2k+2)+(2k+2)
\end{align*}
Note that by uniformity of the cocartesian-ness estimates, allowing other partitions of $\{1,2,3\}$ provides nothing new; hence we have left them out of the above. Hence $l=4k+4$ and we have calculated that the 3-cube $\capX$ is $(4k+4)$-cartesian. By the earlier cases for $n=0,1,2$, together with \cite[3.9]{Ching_Harper} (e.g., \cite[1.8]{Goodwillie_calculus_2}), we have verified that: the 0-subcubes are $(k+1)$-cartesian, the 1-subcubes are $(2k+2)$-cartesian, the 2-subcubes are $(3k+3)$-cartesian, and the 3-subcubes are $(4k+k)$-cartesian. Hence $\capX$ is $(\id+1)(k+1)$-cartesian in $\AlgOY$. And so forth.
\end{proof}

\begin{proof}[Proof of Theorem \ref{thm:main_theorem_for_special_case_of_k_connected}]
Here is the basic idea. The case of 0-cubes is given in Proposition \ref{prop:special_case_of_zero_cube} above. Consider the case of 1-cubes. Suppose $\capX$ is a $\{1\}$-cube in $\AlgOY$ of the form $\capX_\emptyset\rightarrow\capX_{\{1\}}$. 
Assume that $\capX$ is $(\id+1)(k+1)$-cartesian in $\AlgOY$; this means that: the 0-subcubes are $(k+1)$-cartesian and the 1-subcubes are $(2k+2)$-cartesian. Let's verify that the 2-cube of the form $\capX\rightarrow\Loopt_Y\Suspt_Y\capX$ is $(\id+1)(k+1)$-cartesian in $\AlgOY$. It suffices to assume that $\capX$ is a cofibration 1-cube. Let $C$ be the homotopy cofiber of $\capX_\emptyset\rightarrow\capX_{\{1\}}$ in $\AlgOY$ and consider the associated $\infty$-cocartesian 2-cube of the form \eqref{eq:the_2_cube_X_to_C_special_case_0_connected} in $\AlgOY$, where $\capC$ is the indicated 1-face on the bottom. By Proposition \ref{prop:duality_special_case_of_k_connected}, we know that $\capX$ is $(\id(k+2)+k)$-cocartesian; in particular, $\capX$ is $(2k+2)$-cocartesian and hence $C$ is $(2k+2)$-connected (rel. $Y$). Putting it all together, it follows that the vertical maps in diagram \eqref{eq:the_2_cube_X_to_C_special_case_0_connected} are $(k+1)$-connected and the horizontal maps are $(2k+2)$-connected. Here is our strategy: consider the commutative diagram of 2-cubes in $\AlgOY$ of the form \eqref{eq:the_strategy_picture_special_case_0_connected}. Instead of attempting to estimate the cartesian-ness of the 2-cube $(*)$ directly, which seems difficult, we will take an indirect attack and first estimate the cartesian-ness of the 2-cubes $(a),(b),(c)$; then we will use \cite[3.9]{Ching_Harper} (e.g., \cite[1.8]{Goodwillie_calculus_2}) to deduce an estimate for $(*)$. Consider the 2-cube $(a)$. By higher Blakers-Massey \cite[1.7]{Ching_Harper} for $\AlgO$, we know that $(a)$ is $l$-cartesian where $l$ is the minimum of
\begin{align*}
  -2+l_{\{1,2\}}+1 &= -1+\infty\\
  -2+l_{\{1\}}+1+l_{\{2\}}+1 &= (2k+2)+(k+1)
\end{align*}
Hence $l=3k+3$ and we have calculated that the 2-cube $(a)$ is $(3k+3)$-cartesian. Consider the 2-cube $(b)$. The 2-cube $\Suspt_Y\capX\rightarrow\Suspt_Y\capC$ is $\infty$-cocartesian and has the form \eqref{eq:suspension_of_the_2_cube_X_to_C_special_case_0_connected} in $\AlgOY$, where $\Suspt_Y{*_Y}\wequiv{*_Y}$. It follows that the vertical maps in diagram \eqref{eq:suspension_of_the_2_cube_X_to_C_special_case_0_connected} are $(k+2)$-connected and the horizontal maps are $(2k+3)$-connected. By higher Blakers-Massey \cite[1.7]{Ching_Harper} for $\AlgO$, we know that $\Suspt_Y\capX\rightarrow\Suspt_Y\capC$ is $l$-cartesian where $l$ is the minimum of
\begin{align*}
  -2+l_{\{1,2\}}+1 &= -1+\infty\\
  -2+l_{\{1\}}+1+l_{\{2\}}+1 &= (2k+3)+(k+2)
\end{align*}
Hence $l=(3k+5)$ and we have calculated that the 2-cube $\Suspt_Y\capX\rightarrow\Suspt_Y\capC$ is $(3k+5)$-cartesian; therefore the 2-cube $(b)$ is $(3k+4)$-cartesian. Consider the 2-cube (c). We know that $C$ is $(2k+2)$-connected (rel. $Y$) from above, hence by Proposition \ref{prop:special_case_of_zero_cube} the map $C\rightarrow\Loopt_Y\Suspt_Y C$ is $(4k+6)$-connected. This calculation will produce a cartesian-ness estimate for the 2-cube $(c)$. Here is why: the 2-cube (c) has the form \eqref{eq:the_2_cube_labelled_c} in $\AlgOY$, where $\Loopt_Y\Suspt_Y{*_Y}\wequiv {*_Y}$. Taking homotopy fibers horizontally produces the map $\Loopt_Y C\rightarrow\Loopt_Y\Loopt_Y\Suspt_Y C$; this map is $\Loopt_Y$ of the right-hand vertical map. We know $(\#)$ is $(4k+6)$-connected from above; hence the 2-cube $(c)$ is $(4k+5)$-cartesian. Putting it all together, it follows from diagram \eqref{eq:the_strategy_picture_special_case_0_connected} and \cite[3.9]{Ching_Harper} (e.g., \cite[1.8]{Goodwillie_calculus_2}), together with our cartesian-ness estimates for $(a),(b),(c)$,  that the 2-cube $(*)$ of the form \eqref{eq:the_2_cube_labelled_star} in $\AlgOY$, is $(3k+3)$-cartesian. Finally, note that since the top horizontal map in \eqref{eq:the_2_cube_labelled_star} is $(2k+2)$-connected, it follows that the bottom horizontal map is $(2k+2)$-connected; i.e., the bottom horizontal 1-face is $(2k+2)$-cartesian. Hence we have verified that the 2-cube $\capX\rightarrow\Loopt_Y\Suspt_Y\capX$ satisfies: the 0-subcubes are $(k+1)$-cartesian, the 1-subcubes are $(2k+2)$-cartesian, and the 2-subcubes are $(3k+3)$-cartesian. Therefore, the 2-cube of the form $\capX\rightarrow\Loopt_Y\Suspt_Y\capX$ is $(\id+1)(k+1)$-cartesian in $\AlgOY$.

Consider the case of 2-cubes. Suppose $\capX$ is a $\{1,2\}$-cube in $\AlgOY$ of the form \eqref{eq:2_cube_X_in_AlgOY}. Assume that $\capX$ is $(\id+1)(k+1)$-cartesian in $\AlgOY$; this means that: the 0-subcubes are $(k+1)$-cartesian, the 1-subcubes are $(2k+2)$-cartesian, and the 2-subcubes are $(3k+3)$-cartesian. Let's verify that the 3-cube of the form $\capX\rightarrow\Loopt_Y\Suspt_Y\capX$ is $(\id+1)(k+1)$-cartesian in $\AlgOY$. It suffices to assume that $\capX$ is a cofibration 2-cube. Let $C$ be the iterated homotopy cofiber of $\capX$ in $\AlgOY$ and consider the associated $\infty$-cocartesian 3-cube of the form \eqref{eq:the_3_cube_X_to_C_special_case_0_connected} in $\AlgOY$, where $\capC$ is the indicated 2-face on the bottom. By Proposition \ref{prop:duality_special_case_of_k_connected}, we know that $\capX$ is $(\id(k+2)+k)$-cocartesian; in particular, $\capX$ is $(3k+4)$-cocartesian and hence $C$ is $(3k+4)$-connected (rel. $Y$). Putting it all together, it follows that the vertical maps in diagram \eqref{eq:the_3_cube_X_to_C_special_case_0_connected} are $(k+1)$-connected, the horizontal maps in the top 2-face (i.e., the 1-faces of $\capX$) are $(2k+2)$-connected, and the top 2-face is $(3k+4)$-cocartesian. We would like to estimate the cocartesian-ness of the back 2-face of the form \eqref{eq:back_two_face_of_the_form}. The upper horizontal map is $(2k+2)$-connected and the lower horizontal map is $\infty$-connected (i.e., a weak equivalence), hence by \cite[3.8]{Ching_Harper} the back 2-face in \eqref{eq:the_3_cube_X_to_C_special_case_0_connected} is $(2k+3)$-cocartesian in $\AlgOY$. Similarly, the left-hand 2-face in \eqref{eq:the_3_cube_X_to_C_special_case_0_connected} is $(2k+3)$-cocartesian in $\AlgOY$. Here is our strategy: consider the commutative diagram of 3-cubes in $\AlgOY$ of the form \eqref{eq:the_strategy_picture_special_case_0_connected}. Instead of attempting to estimate the cartesian-ness of the 3-cube $(*)$ directly, which seems difficult, we will take an indirect attack and first estimate the cartesian-ness of the 3-cubes $(a),(b),(c)$; then we will use \cite[3.9]{Ching_Harper} (e.g., \cite[1.8]{Goodwillie_calculus_2}) to deduce an estimate for $(*)$. Consider the 3-cube $(a)$. By higher Blakers-Massey \cite[1.7]{Ching_Harper} for $\AlgO$, we know that $(a)$ is $l$-cartesian where $l$ is the minimum of
\begin{align*}
  -3+l_{\{1,2,3\}}+1 &= -2+\infty\\
  -3+l_{\{1,2\}}+1+l_{\{3\}}+1 &= -1+(3k+4)+(k+1)\\
  -3+l_{\{1,3\}}+1+l_{\{2\}}+1 &= -1+(2k+3)+(2k+2)\\
  -3+l_{\{2,3\}}+1+l_{\{1\}}+1 &= -1+(2k+3)+(2k+2)\\
  -3+l_{\{1\}}+1+l_{\{2\}}+1+l_{\{3\}}+1 &= (2k+2)+(2k+2)+(k+1)
\end{align*}
Hence $l=(4k+4)$ and we have calculated that the 3-cube $(a)$ is $(4k+4)$-cartesian. Consider the 3-cube $(b)$. The 3-cube $\Suspt_Y\capX\rightarrow\Suspt_Y\capC$ is $\infty$-cocartesian in $\AlgOY$, where $\Suspt_Y{*_Y}\wequiv{*_Y}$. By higher Blakers-Massey \cite[1.7]{Ching_Harper} for $\AlgO$, we know that $\Suspt_Y\capX\rightarrow\Suspt_Y\capC$ is $l$-cartesian where $l$ is the minimum of
\begin{align*}
  -3+l_{\{1,2,3\}}+1 &= -2+\infty\\
  -3+l_{\{1,2\}}+1+l_{\{3\}}+1 &= -1+(3k+5)+(k+2)\\
  -3+l_{\{1,3\}}+1+l_{\{2\}}+1 &= -1+(2k+4)+(2k+3)\\
  -3+l_{\{2,3\}}+1+l_{\{1\}}+1 &= -1+(2k+4)+(2k+3)\\
  -3+l_{\{1\}}+1+l_{\{2\}}+1+l_{\{3\}}+1 &= (2k+3)+(2k+3)+(k+2)
\end{align*}
Hence $l=(4k+6)$ and we have calculated that the 3-cube $\Suspt_Y\capX\rightarrow\Suspt_Y\capC$ is $(4k+6)$-cartesian; therefore the 3-cube $(b)$ is $(4k+5)$-cartesian. Consider the 3-cube (c). We know that $C$ is $(3k+4)$-connected (rel. $Y$) from above, hence by Proposition \ref{prop:special_case_of_zero_cube} the map $C\rightarrow\Loopt_Y\Suspt_Y C$ is $(6k+10)$-connected. This calculation will produce a cartesian-ness estimate for the 3-cube $(c)$. Here is why: the 3-cube (c) has the form \eqref{eq:3_cube_c_has_the_form} in $\AlgOY$, where $\Loopt_Y\Suspt_Y{*_Y}\wequiv {*_Y'}$. Taking homotopy fibers (twice) in $\AlgOY$ produces the map $\Loopt^2_Y C\rightarrow\Loopt^2_Y\Loopt_Y\Suspt_Y C$; this map is $\Loopt^2_Y$ of the right-hand vertical map $(\#)$. We know $(\#)$ is $(6k+10)$-connected from above; hence the 3-cube $(c)$ is $(6k+8)$-cartesian. Putting it all together, it follows from diagram \eqref{eq:the_strategy_picture_special_case_0_connected} and \cite[3.9]{Ching_Harper} (e.g., \cite[1.8]{Goodwillie_calculus_2}), together with our cartesian-ness estimates for $(a),(b),(c)$,  that the 3-cube $(*)$ of the form
\begin{align*}
\xymatrix{
  \capX:\ar[d]_-{(*)} & \capX_\emptyset\ar[r]\ar[d] & \capX_{\{1\}}\ar[d]\\
  \Loopt_Y\Suspt_Y\capX: & \Loopt_Y\Suspt_Y\capX_\emptyset\ar[r] & 
  \Loopt_Y\Suspt_Y\capX_{\{1\}}
}
\end{align*}
in $\AlgOY$, is $(4k+4)$-cartesian. Let's calculate a cartesian-ness estimate for the 2-subcube $\Loopt_Y\Suspt_Y\capX$ of $(*)$. We know that $\capX$ is $(\id(k+2)+k)$-cocartesian in $\AlgOY$ from above, hence $\Suspt_Y\capX$ is $(\id(k+2)+k+1)$-cocartesian in $\AlgOY$. By higher Blakers-Massey \cite[1.7]{Ching_Harper} for $\AlgO$, we know that $\Suspt_Y\capX$ is $l$-cartesian where $l$ is the minimum of
\begin{align*}
  -2+l_{\{1,2\}}+1 &= -1+(3k+5)\\
  -2+l_{\{1\}}+1+l_{\{2\}}+1 &= (2k+3)+(2k+3)
\end{align*}
 Hence $l=(3k+4)$ and we have calculated that the 2-cube $\Suspt_Y\capX$ is $(3k+4)$-cartesian; therefore the 2-cube $\Loopt_Y\Suspt_Y\capX$ is $(3k+3)$-cartesian. Hence we have verified that the 3-cube $\capX\rightarrow\Loopt_Y\Suspt_Y\capX$ satisfies: the 0-subcubes are $(k+1)$-cartesian, the 1-subcubes are $(2k+2)$-cartesian, the 2-subcubes are $(3k+3)$-cartesian, and the 3-subcubes are $(4k+4)$-cartesian. Therefore, the 3-cube of the form $\capX\rightarrow\Loopt_Y\Suspt_Y\capX$ is $(\id+1)(k+1)$-cartesian in $\AlgOY$. And so forth.
\end{proof}

\begin{thm}
\label{thm:main_theorem_homotopical_estimates}
Let $k\geq 0$ and $1\leq r\leq\infty$. Let $W$ be a finite set and $\capX$ a $W$-cube in $\AlgOY$. Let $n=|W|$. If the $n$-cube $\capX$ is $(\id+1)(k+1)$-cartesian in $\AlgOY$, then so is the $(n+1)$-cube of the form $\capX\rightarrow\Loopt_Y^r\Suspt_Y^r\capX$.
\end{thm}

\begin{proof}
These estimates were worked out in \cite{Blomquist} for the special case of $Y=*$. The detailed proof above of Theorem \ref{thm:main_theorem_for_special_case_of_k_connected} in the case of $r=1$ shows why these estimates remain true in the context of $\capO$-algebras centered at $Y$; intuitively, $r=1$ is the most restrictive situation in terms of estimates (e.g,. the comparison map in all other cases factors through the comparison map for $r=1$). In the case of $r=\infty$, several of the estimate steps are easier since $\Suspt^\infty_Y$ preserves cocartesian-ness, $\Loopt^\infty_Y$ preserves cartesian-ness, and the stable estimates in \cite[3.10]{Ching_Harper} become available for each estimate step following the application of $\Suspt^\infty_Y$.
\end{proof}

\section{Homotopical resolutions}
\label{sec:proofs_of_main_results}

The purpose of this section is to introduce the homotopical resolutions studied in \cite{Blomquist} and to prove Theorems \ref{main_theorem_A}, \ref{main_theorem_B}, and \ref{main_theorem_C}. These kinds of homotopical resolutions are studied in \cite{Blumberg_Riehl}; see also \cite{Blomquist_Harper, Ching_Harper_derived_Koszul_duality}. There are adjunctions of the form

\begin{align}
\xymatrix{
  \AlgOY\ar@<0.5ex>[r]^-{\Susp^r_Y} &
  \AlgOY\ar@<0.5ex>[l]^-{\Loop^r_Y}
}
\quad\quad
\xymatrix{
  \AlgOY\ar@<0.5ex>[r]^-{\Susp^\infty_Y} &
  \SpectraN(\AlgOY)\ar@<0.5ex>[l]^-{\Loop^\infty_Y}
}
\quad\quad
(r\geq 1)
\end{align}
with left adjoint on top. Here, $\SpectraN(\AlgOY)$ denotes Hovey spectra \cite{Hovey_spectra} on $\AlgOY$; we are using that left Bousfield localization constructions produce semi-model categories in situations where left properness is not available (together with the fact that in $\AlgO$, and hence in $\AlgOY$, the generating cofibrations and acyclic cofibrations have cofibrant domains); see \cite{Batanin_White, Carmona, Goerss_Hopkins_moduli_problems, Harper_Zhang}. If we iterate the comparison map $\id\rightarrow \Loop^r_Y\Susp^r_Y$ and evaluate on $X\in\AlgOY$, this builds a cosimplicial resolution of $X$ with respect to $\Loop^r_Y\Susp^r_Y$ of the form $(1\leq r\leq \infty)$
\begin{align*}
\xymatrix{
  \id(X)\ar[r] &
  \Loop^r_Y\Susp^r_Y(X)\ar@<-0.5ex>[r]\ar@<0.5ex>[r] &
  (\Loop^r_Y\Susp^r_Y)^2(X)
  \ar@<-1.0ex>[r]\ar[r]\ar@<1.0ex>[r] &
  (\Loop^r_Y\Susp^r_Y)^3(X)\cdots
  }
\end{align*}
showing only the coface maps. This resolution is not what we want---it is not homotopy meaningful.

The basic idea is to turn this into a derived (homotopy meaningful) resolution by appropriately inserting fibrant replacements and cofibrant replacements. In more detail, let $r\geq 1$. Denote by $\function{\varepsilon}{Q}{\id}$ and $\function{m}{Q}{QQ}$ the counit and comultiplication maps of the cofibrant replacement comonad $Q$ on $\AlgOY$ (see \cite[6.1]{Blumberg_Riehl}) and define $\Suspt^r_Y:=\Susp^r_Y Q$ (resp. $\Suspt^\infty_Y:=\Susp^\infty_Y Q$). Denote by $\function{\eta}{\id}{\Phi}$ and $\function{m}{\Phi\Phi}{\Phi}$ the unit and multiplication maps of the fibrant replacement monad $\Phi$ on $\AlgOY$ (see \cite[6.1]{Blumberg_Riehl}) and define 
$\Loopt^r_Y:=\Loop^r_Y \Phi$. Similarly, denote by $\function{\eta}{\id}{F}$ and $\function{m}{FF}{F}$ the unit and multiplication maps of the fibrant replacement monad $F$ on $\SpectraN(\AlgOY)$ (see \cite[6.1]{Blumberg_Riehl}) and define $\Loopt_Y^\infty:=\Loop^\infty_Y F$.

\begin{rem} Since $\SpectraN(\AlgOY)$ is only equipped with a cofibrantly generated simplicial semi-model structure ($\AlgOY$ is not left proper, in general, so left Bousfield localization only produces a semi-model structure, without further work)---see, for instance, \cite{Batanin_White, Carmona, Goerss_Hopkins_moduli_problems, Harper_Zhang}---this simply means that $F,\Loopt^\infty_Y$ are only homotopy meaningful when evaluated on cofibrant objects. In other words, small object arguments (and their generalizations that produce monads) in $\SpectraN(\AlgOY)$ only behave as fibrant replacements when evaluated on cofibrant objects; we note that this is always satisfied in the homotopical resolutions below.
\end{rem}

It follows easily that we can build a cosimplicial resolution of $X$ with respect to $\Loopt^r_Y\Suspt^r_Y$ of the form $(1\leq r\leq \infty)$
\begin{align*}
\xymatrix{
  \id(X)\ar[r] &
  \Loopt^r_Y\Suspt^r_Y(X)\ar@<-0.5ex>[r]\ar@<0.5ex>[r] &
  (\Loopt^r_Y\Suspt^r_Y)^2(X)
  \ar@<-1.0ex>[r]\ar[r]\ar@<1.0ex>[r] &
  (\Loopt^r_Y\Suspt^r_Y)^3(X)\cdots
  }
\end{align*}
in $\AlgOY$, which is given precisely by the homotopical resolution
\begin{align}
\label{eq:cosimplicial_resolution_derived}
\xymatrix{
  Q(X)\ar[r] &
  Q\Loopt^r_Y\Suspt^r_Y(X)\ar@<-0.5ex>[r]\ar@<0.5ex>[r] &
  Q(\Loopt^r_Y\Suspt^r_Y)^2(X)
  \ar@<-1.0ex>[r]\ar[r]\ar@<1.0ex>[r] &
  Q(\Loopt^r_Y\Suspt^r_Y)^3(X)\cdots
  }
\end{align}
in $\AlgOY$ (showing only the coface maps); see, for instance, \cite{Blomquist, Blumberg_Riehl}. If $X\in\AlgOY$, the Bousfield-Kan completion $X^\wedge_{\Loopt^r_Y\Suspt^r_Y}$ of $X$ with respect to $\Loopt^r_Y\Suspt^r_Y$ is the homotopy limit
\begin{align*}
  X^\wedge_{\Loopt^r_Y\Suspt^r_Y}:=\holim\nolimits_\Delta Q(\Loopt^r_Y\Suspt^r_Y)^{\bullet+1}(X)
\end{align*}
of the cosimplicial resolution indicated on the right-hand side of \eqref{eq:cosimplicial_resolution_derived}. To get our hands on this construction, we filter $\Delta$ by its subcategories $\Delta^{\leq n}\subset\Delta$, $n\geq 0$, which leads to the following (e.g., \cite{Ching_Harper_derived_Koszul_duality}). Define 
\begin{align*}
  (\Loopt^r_Y\Suspt^r_Y)_n:=
  \holim\nolimits_{\Delta^{\leq n}}Q(\Loopt^r_Y\Suspt^r_Y)^{\bullet+1},
  \quad\quad
  n\geq 0
\end{align*}
It follows that the $\Loopt^r_Y\Suspt^r_Y$-completion of $X$ is weakly equivalent to
\begin{align}
\label{eq:the_indicated_tower}
  X^\wedge_{\Loopt^r_Y\Suspt^r_Y}\wequiv\holim
  \Bigl(
  (\Loopt^r_Y\Suspt^r_Y)_0(X)\leftarrow(\Loopt^r_Y\Suspt^r_Y)_1(X)
  \leftarrow(\Loopt^r_Y\Suspt^r_Y)_2(X)\leftarrow\cdots
  \Bigr)
\end{align}
the homotopy limit of the indicated tower; here, $(\Loopt^r_Y\Suspt^r_Y)_0(X)\wequiv \Loopt^r_Y\Suspt^r_Y(X)$.

\begin{proof}[Proof of Theorem \ref{main_theorem_A}]
Here is the basic idea. To verify that $X\wequiv X^\wedge_{\Loopt^r_Y\Suspt^r_Y}$, it suffices to verify that the map of the form
\begin{align*}
  X\wequiv QX\xrightarrow{(*)_n} (\Loopt^r_Y\Suspt^r_Y)_n(X)
\end{align*}
into the $n$-th stage of the completion tower in \eqref{eq:the_indicated_tower} has connectivity strictly increasing with $n$. The connectivity of the map $(*)_n$ is the same as the cartesian-ness of the coface $(n+1)$-cube (\cite[5.20]{Blomquist_Harper}) of the coaugmented cosimplicial resolution \eqref{eq:cosimplicial_resolution_derived} which we know (Theorem \ref{thm:main_theorem_homotopical_estimates}) is $((n+1)+1)(0+1)=n+2$; hence the map $(*)_n$ is $(n+2)$-connected which completes the proof.
\end{proof}

\begin{proof}[Proof of Theorem \ref{main_theorem_B}]
Here is the basic idea. We will follow the proof ideas in \cite{Schonsheck_fibration}  and  exploit the estimates in Theorem \ref{thm:main_theorem_homotopical_estimates}.  We start with the fibration sequence in $\AlgOY$ of the form $F\rightarrow E\rightarrow B$ and resolve the $E,B$ terms
\begin{align*}
\xymatrix{
  F\ar[r]\ar[d] &
  \tilde{F}^0\ar[d]\ar@<-0.5ex>[r]\ar@<0.5ex>[r] &
  \tilde{F}^1\ar[d]
  \ar@<-1.0ex>[r]\ar[r]\ar@<1.0ex>[r] &
  \tilde{F}^2\ar[d]\cdots\\
  E\ar[r]\ar[d] &
  \Loopt^r_Y\Suspt^r_Y(E)\ar[d]\ar@<-0.5ex>[r]\ar@<0.5ex>[r] &
  (\Loopt^r_Y\Suspt^r_Y)^2(E)
  \ar@<-1.0ex>[r]\ar[r]\ar@<1.0ex>[r]\ar[d] &
  (\Loopt^r_Y\Suspt^r_Y)^3(E)\cdots\ar[d]\\
  B\ar[r] &
  \Loopt^r_Y\Suspt^r_Y(B)\ar@<-0.5ex>[r]\ar@<0.5ex>[r] &
  (\Loopt^r_Y\Suspt^r_Y)^2(B)
  \ar@<-1.0ex>[r]\ar[r]\ar@<1.0ex>[r] &
  (\Loopt^r_Y\Suspt^r_Y)^3(B)\cdots
  }
\end{align*}
by their cosimplicial resolutions with respect to $\Loopt^r_Y\Suspt^r_Y$; this produces the bottom two rows of the indicated form (for notational convenience, we omit writing the $Q$). Taking homotopy fibers vertically in $\AlgOY$ produces the top horizontal row of the form $F\rightarrow\tilde{F}$; in particular, each of the rows is a coaugmented cosimplicial diagram in $\AlgOY$, and the columns are homotopy fiber sequences in $\AlgOY$. Since homotopy fibers commute with homotopy limits, together with the assumption that $E,B$ are 0-connected (rel. $Y$) and Theorem \ref{main_theorem_A}, we know that $F\wequiv\holim\nolimits_\Delta\tilde{F}$. 

The next step is to get $\Loopt^r_Y\Suspt^r_Y$-completion into the picture. Resolving each term in the upper row above by its cosimplicial resolution with respect to $\Loopt^r_Y\Suspt^r_Y$ produces a diagram of the form
\begin{align*}
\xymatrix{
  (\Loopt^r_Y\Suspt^r_Y)^3F\ar[r]^-{(\#)} &
  (\Loopt^r_Y\Suspt^r_Y)^3\tilde{F}^0\ar@<-0.5ex>[r]\ar@<0.5ex>[r] &
  (\Loopt^r_Y\Suspt^r_Y)^3\tilde{F}^1
  \ar@<-1.0ex>[r]\ar[r]\ar@<1.0ex>[r] &
  (\Loopt^r_Y\Suspt^r_Y)^3\tilde{F}^2\cdots\\
  (\Loopt^r_Y\Suspt^r_Y)^2F\ar[r]^-{(\#)}\ar@<-1.0ex>[u]\ar[u]\ar@<1.0ex>[u] &
  (\Loopt^r_Y\Suspt^r_Y)^2\tilde{F}^0\ar@<-0.5ex>[r]\ar@<0.5ex>[r]\ar@<-1.0ex>[u]\ar[u]\ar@<1.0ex>[u] &
  (\Loopt^r_Y\Suspt^r_Y)^2\tilde{F}^1\ar@<-1.0ex>[u]\ar[u]\ar@<1.0ex>[u]
  \ar@<-1.0ex>[r]\ar[r]\ar@<1.0ex>[r] &
  (\Loopt^r_Y\Suspt^r_Y)^2\tilde{F}^2\cdots\ar@<-1.0ex>[u]\ar[u]\ar@<1.0ex>[u]\\
  (\Loopt^r_Y\Suspt^r_Y)F\ar[r]^-{(\#)}\ar@<-0.5ex>[u]\ar@<0.5ex>[u] &
  (\Loopt^r_Y\Suspt^r_Y)\tilde{F}^0\ar@<-0.5ex>[r]\ar@<0.5ex>[r]\ar@<-0.5ex>[u]\ar@<0.5ex>[u] &
  (\Loopt^r_Y\Suspt^r_Y)\tilde{F}^1
  \ar@<-1.0ex>[r]\ar[r]\ar@<1.0ex>[r]\ar@<-0.5ex>[u]\ar@<0.5ex>[u] &
  (\Loopt^r_Y\Suspt^r_Y)\tilde{F}^2\cdots\ar@<-0.5ex>[u]\ar@<0.5ex>[u]\\
  F\ar[r]^-{(\#)}\ar[u] &
  \tilde{F}^0\ar@<-0.5ex>[r]\ar@<0.5ex>[r]\ar[u]_-{(**)} &
  \tilde{F}^1\ar[u]_-{(**)}
  \ar@<-1.0ex>[r]\ar[r]\ar@<1.0ex>[r] &
  \tilde{F}^2\cdots\ar[u]_-{(**)}
  }
\end{align*}
where each column (resp. row) is a coaugmented cosimplicial diagram in $\AlgOY$. Since $E,B$ are 0-connected (rel. $Y$) by assumption, we know from Theorem \ref{thm:main_theorem_homotopical_estimates} that the coface $(n+1)$-cubes (\cite[5.20]{Blomquist_Harper}) associated to the coaugmented cosimplical resolutions of the form
\begin{align*}
  E&\rightarrow(\Loopt^r_Y\Suspt^r_Y)^{\bullet+1}E\\
  B&\rightarrow(\Loopt^r_Y\Suspt^r_Y)^{\bullet+1}B
\end{align*}
are $(\id+1)$-cartesian for each $n\geq -1$. Hence it follows, by several applications of \cite[3.8]{Ching_Harper} and \cite[1.18]{Goodwillie_calculus_2} that the coface $(n+1)$-cube associated to $F\rightarrow\tilde{F}$ is $\id$-cartesian for each $n\geq -1$. An easy consequence of higher Blakers-Massey (and its dual) \cite[1.7, 1.11]{Ching_Harper} for $\AlgO$ is that $\Loopt^r_Y\Suspt^r_Y$ preserves $\id$-cartesian $(n+1)$-cubes for each $n\geq -1$. Hence it follows that the coface $(n+1)$-cubes (\cite[5.20]{Blomquist_Harper}) associated to the coaugmented cosimplicial diagrams
\begin{align*}
  (\Loopt^r_Y\Suspt^r_Y)^k F\rightarrow 
  (\Loopt^r_Y\Suspt^r_Y)^k\tilde{F},\quad\quad
  k\geq 0
\end{align*}
are $\id$-cartesian for each $n\geq -1$. Therefore, each of the maps 
\begin{align}
\label{eq:the_estimates_we_need_for_sharp_stuff}
  (\Loopt^r_Y\Suspt^r_Y)^k F\xrightarrow{(\#)_n}
  \holim\nolimits_{\Delta^{\leq n}}
  (\Loopt^r_Y\Suspt^r_Y)^k\tilde{F},\quad\quad
  k\geq 0
\end{align}
is $(n+1)$-connected; hence each of the maps $(\#)$ induces a weak equivalence on $\holim_\Delta$. In other words, applying $\holim_\Delta$ horizontally produces the left-hand column, and therefore, subsequently applying $\holim_\Delta$ vertically produces $F^\wedge_{\Loopt_Y^r\Suspt_Y^r}$. What about the other way? By formal arguments (i.e., $\Loopt^r_Y$ commutes with homotopy fibers), the $(**)$ columns have extra codegeneracy maps $s^{-1}$ \cite[6.2]{Dwyer_Miller_Neisendorfer}; hence applying $\holim_\Delta$ vertically induces a weak equivalence \cite[3.16]{Dror_Dwyer_long_homology} with the $\tilde{F}$ cosimplicial diagram, and therefore, subsequently applying $\holim_\Delta$ horizontally produces $F$ (i.e., we know that $F\wequiv\holim\nolimits_\Delta\tilde{F}$ as noted above, or by the connectivities in \eqref{eq:the_estimates_we_need_for_sharp_stuff} with $k=0$). Hence we have verified that $F\wequiv F^\wedge_{\Loopt_Y^r\Suspt_Y^r}$. 

The more general case of an $\infty$-cartesian 2-cube is nearly identical; constructing $F\rightarrow\tilde{F}$ by taking homotopy pullbacks instead of homotopy fibers, it follows, by several applications of \cite[3.8]{Ching_Harper} and \cite[1.18]{Goodwillie_calculus_2}, that the coface $(n+1)$-cube associated to $F\rightarrow\tilde{F}$ is $\id$-cartesian for each $n\geq -1$, and the above arguments complete the proof.
\end{proof}

\begin{proof}[Proof of Theorem \ref{main_theorem_C}]
Here is the basic idea. We will follow the proof ideas in \cite{Schonsheck_TQ} and  exploit the estimates in Theorem \ref{thm:main_theorem_homotopical_estimates}. Start with the tower \eqref{eq:the_indicated_tower} associated with the cosimplicial resolution of the identity functor on $\AlgOY$ with respect to $\Loopt^\infty_Y\Suspt^\infty_Y$, resolve each functor in the tower
\begin{align}
\label{eq:tower_of_towers_diagram}
\xymatrix{
  P_3^Y(\Loopt^\infty_Y\Suspt^\infty_Y)_0(X)\ar[d] &
  P_3^Y(\Loopt^\infty_Y\Suspt^\infty_Y)_1(X)\ar[l]\ar[d] &
  P_3^Y(\Loopt^\infty_Y\Suspt^\infty_Y)_2(X)\ar[l]\ar[d]\cdots\\
  P_2^Y(\Loopt^\infty_Y\Suspt^\infty_Y)_0(X)\ar[d] &
  P_2^Y(\Loopt^\infty_Y\Suspt^\infty_Y)_1(X)\ar[l]\ar[d] &
  P_2^Y(\Loopt^\infty_Y\Suspt^\infty_Y)_2(X)\ar[l]\ar[d]\cdots\\
  P_1^Y(\Loopt^\infty_Y\Suspt^\infty_Y)_0(X) &
  P_1^Y(\Loopt^\infty_Y\Suspt^\infty_Y)_1(X)\ar[l] &
  P_1^Y(\Loopt^\infty_Y\Suspt^\infty_Y)_2(X)\ar[l]\cdots\\
  (\Loopt^\infty_Y\Suspt^\infty_Y)_0(X)\ar@{.>}[u] &
  (\Loopt^\infty_Y\Suspt^\infty_Y)_1(X)\ar[l]\ar@{.>}[u] &
  (\Loopt^\infty_Y\Suspt^\infty_Y)_2(X)\ar[l]\ar@{.>}[u]\cdots
}
\end{align}
by its Taylor tower \cite{Goodwillie_calculus_3} centered at $Y$, and evaluate the result on $X$. We know, by Theorem \ref{thm:main_theorem_homotopical_estimates}, that the map $\id\rightarrow(\Loopt^\infty_Y\Suspt^\infty_Y)_n$ satisfies $O_{n+1}(0,1)$ (\cite[1.2]{Goodwillie_calculus_3}) for every $n\geq 0$ (for notational convenience, we omit writing the $Q$); in particular, the identity functor and $(\Loopt^\infty_Y\Suspt^\infty_Y)_n$ agree to order $(n+1)$ via this map; this implies (\cite[1.6]{Goodwillie_calculus_3}) that the maps
\begin{align*}
  P_{n+1}^Y(\id)(X)&\xrightarrow{\wequiv} P_{n+1}^Y(\Loopt^\infty_Y\Suspt^\infty_Y)_n(X)\\
  P_{n}^Y(\id)(X)&\xrightarrow{\wequiv} P_{n}^Y(\Loopt^\infty_Y\Suspt^\infty_Y)_n(X)\\
  \cdots\\
  P_1^Y(\id)(X)&\xrightarrow{\wequiv} P_1^Y(\Loopt^\infty_Y\Suspt^\infty_Y)_n(X)
\end{align*}
are weak equivalences for every $n\geq 0$. This means that the rows above the dotted arrows in \eqref{eq:tower_of_towers_diagram} are eventually constant and hence applying $\holim$ horizontally produces the tower $\{P_n^Y(\id)(X)\}$, and therefore, subsequently applying $\holim$ vertically produces $P_\infty^Y(\id)(X)$. What about the other way? By assumption, $*_Y\rightarrow \Suspt_Y^\infty X$ in $\SpectraN(\AlgOY)$ is 0-connected (Remark \ref{rem:about_stuff_useful}). It follows, by iteratively applying higher Blakers-Massey (and its dual) \cite[1.7, 1.11]{Ching_Harper} for $\AlgO$,
\begin{align*}
  (\Loopt^\infty_Y\Suspt^\infty_Y)^k(X)\xrightarrow{(\#)_1}
  T_1^Y(\Loopt^\infty_Y\Suspt^\infty_Y)^k(X)\rightarrow
  T_1^Y(T_1^Y(\Loopt^\infty_Y\Suspt^\infty_Y)^k)(X)\rightarrow\cdots
\end{align*}
that the maps $(\#)_1$ are $3$-connected (and the indicated subsequent maps are even higher connected) for each $k\geq 2$. Hence the maps
\begin{align*}
  (\Loopt^\infty_Y\Suspt^\infty_Y)^k(X)\xrightarrow{(*)_1}
  P_1^Y(\Loopt^\infty_Y\Suspt^\infty_Y)^k(X)
\end{align*}
are $3$-connected for each $k\geq 2$; if $k=1$, no estimates are required: the map $(*)_1$ is $\infty$-connected since $\Loopt^\infty_Y\Suspt^\infty_Y\wequiv P_1^Y(\id)$ and hence $P_1^Y(\id)\rightarrow P_1^Y(P_1^Y(\id))$ is a weak equivalence. That was a useful warmup. More generally, by iteratively applying higher Blakers-Massey (and its dual) \cite[1.7, 1.11]{Ching_Harper} for $\AlgO$, it follows that $(n\geq 1)$
\begin{align*}
  (\Loopt^\infty_Y\Suspt^\infty_Y)^k(X)\xrightarrow{(\#)_n}
  T_n^Y(\Loopt^\infty_Y\Suspt^\infty_Y)^k(X)\rightarrow
  T_n^Y(T_n^Y(\Loopt^\infty_Y\Suspt^\infty_Y)^k)(X)\rightarrow\cdots
\end{align*}
the maps $(\#)_n$ are $(n+2)$-connected (and the indicated subsequent maps are even higher connected) for each $k\geq 2$. Hence the maps
\begin{align*}
  (\Loopt^\infty_Y\Suspt^\infty_Y)^k(X)\xrightarrow{(*)_n}
  P_n^Y(\Loopt^\infty_Y\Suspt^\infty_Y)^k(X),\quad\quad
  n\geq 1
\end{align*}
are $(n+2)$-connected for each $k\geq 2$; if $k=1$, no estimates are required: the map $(*)_n$ is $\infty$-connected since $\Loopt^\infty_Y\Suspt^\infty_Y\wequiv P_1^Y(\id)$ and hence $P_1^Y(\id)\rightarrow P_n^Y(P_1^Y(\id))$ is a weak equivalence. It now follows by iteratively applying \cite[3.8]{Ching_Harper} (e.g., \cite[1.6]{Goodwillie_calculus_2}) that the maps
\begin{align*}
  (\Loopt^\infty_Y\Suspt^\infty_Y)_k(X)\xrightarrow{(**)_n}
  P_{n+k}^Y(\Loopt^\infty_Y\Suspt^\infty_Y)_k(X),\quad\quad
  n\geq 1
\end{align*}
are $(n+2)$-connected for each $k\geq 1$; if $k=0$, no estimates are required: the map $(**)_n$ is $\infty$-connected since $\Loopt^\infty_Y\Suspt^\infty_Y\wequiv P_1^Y(\id)$ and hence $P_1^Y(\id)\rightarrow P_n^Y(P_1^Y(\id))$ is a weak equivalence. Let's illustrate the case of $k=1$: the map $(**)_n$ fits into a 3-cube of the form

\begin{align}
\label{eq:3_cube_c_has_the_form_for_final_proof}
\xymatrix{
(\Loopt^\infty_Y\Suspt^\infty_Y)_1 X\ar[dd]\ar[rr]\ar[dr]^-{(**)_n} &&
\Loopt^\infty_Y\Suspt^\infty_Y X\ar'[d][dd]\ar[dr]^{\wequiv}\\
&P_{n+1}^Y(\Loopt^\infty_Y\Suspt^\infty_Y)_1 X\ar[dd]\ar[rr] &&
P_{n+1}^Y(\Loopt^\infty_Y\Suspt^\infty_Y) X\ar[dd]\\
\Loopt^\infty_Y\Suspt^\infty_Y X\ar'[r][rr]\ar[dr]^{\wequiv} &&
(\Loopt^\infty_Y\Suspt^\infty_Y)^2 X\ar[dr]^-{(*)_{n+1}}\\
&P_{n+1}^Y(\Loopt^\infty_Y\Suspt^\infty_Y) X\ar[rr] &&
P_{n+1}^Y(\Loopt^\infty_Y\Suspt^\infty_Y)^2 X
}
\end{align}
The back 2-face is $\infty$-cartesian by definition, hence the front 2-face is $\infty$-cartesian ($P_{n+1}^Y$ commutes (\cite[1.7]{Goodwillie_calculus_3}) with such $\holim$'s). The 3-cube is therefore $\infty$-cartesian by \cite[3.8]{Ching_Harper} (e.g., \cite[1.6]{Goodwillie_calculus_2}). We know from above that the map $(*)_{n+1}$ is $(n+3)$-connected, hence by \cite[3.8]{Ching_Harper} the right-hand 2-face is $(n+2)$-cartesian. Since the 3-cube is $\infty$-cartesian, it follows that the left-hand 2-face is $(n+2)$-cartesian. Therefore, it follows, by another application of \cite[3.8]{Ching_Harper} (e.g., \cite[1.6]{Goodwillie_calculus_2}) that the map $(**)_n$ is $(n+2)$-connected. The other cases are similar. The upshot is: it follows that applying $\holim$ vertically (above the dotted arrows) in \eqref{eq:tower_of_towers_diagram} produces the bottom horizontal tower, and therefore, subsequently applying $\holim$ horizontally produces $X^\wedge_{\Loopt_Y^\infty\Suspt_Y^\infty}$. Hence we have verified that $P^Y_\infty(\id)(X)\wequiv X^\wedge_{\Loopt_Y^\infty\Suspt_Y^\infty}$.
\end{proof}

\bibliographystyle{plain}
\bibliography{FunctorCalcCompletion}

\begin{thebibliography}{10}

\bibitem{Arone_Kankaanrinta}
G.~Arone and M.~Kankaanrinta.
\newblock A functorial model for iterated {S}naith splitting with applications
  to calculus of functors.
\newblock In {\em Stable and unstable homotopy ({T}oronto, {ON}, 1996)},
  volume~19 of {\em Fields Inst. Commun.}, pages 1--30. Amer. Math. Soc.,
  Providence, RI, 1998.

\bibitem{Batanin_White}
M.~Batanin and D.~White.
\newblock Left {B}ousfield localization without left properness.
\newblock {\em J. Pure Appl. Algebra}, 228(6):Paper No. 107570, 23, 2024.

\bibitem{Bauer_Johnson_McCarthy}
K.~Bauer, B.~Johnson, and R.~McCarthy.
\newblock Cross effects and calculus in an unbased setting.
\newblock {\em Trans. Amer. Math. Soc.}, 367(9):6671--6718, 2015.
\newblock With an appendix by R. Eldred.

\bibitem{Beardsley_Lawson}
J.~Beardsley and T.~Lawson.
\newblock Skeleta and categories of algebras.
\newblock {\em \\ \verb=arXiv:2110.09595 [math.AT]=}, 2021.

\bibitem{Blomquist}
J.~R. Blomquist.
\newblock Iterated delooping and desuspension of structured ring spectra.
\newblock {\em \\ \verb=arXiv:1910.12442 [math.AT]=}, 2019.

\bibitem{Blomquist_Harper_integral_chains}
J.~R. Blomquist and J.~E. Harper.
\newblock Integral chains and {B}ousfield-{K}an completion.
\newblock {\em Homology Homotopy Appl.}, 21(2):29--58, 2019.

\bibitem{Blomquist_Harper}
J.~R. Blomquist and J.~E. Harper.
\newblock Higher stabilization and higher {F}reudenthal suspension.
\newblock {\em Trans. Amer. Math. Soc.}, 375(11):8193--8240, 2022.

\bibitem{Blumberg_Riehl}
A.~J. Blumberg and E.~Riehl.
\newblock Homotopical resolutions associated to deformable adjunctions.
\newblock {\em Algebr. Geom. Topol.}, 14(5):3021--3048, 2014.

\bibitem{Bousfield_Kan}
A.~K. Bousfield and D.~M. Kan.
\newblock {\em Homotopy limits, completions and localizations}.
\newblock Lecture Notes in Mathematics, Vol. 304. Springer-Verlag, Berlin,
  1972.

\bibitem{Carmona}
V.~Carmona.
\newblock When {B}ousfield localizations and homotopy idempotent functors meet
  again.
\newblock {\em Homology Homotopy Appl.}, 25(2):187--218, 2023.

\bibitem{Ching_Harper}
M.~Ching and J.~E. Harper.
\newblock Higher homotopy excision and {B}lakers-{M}assey theorems for
  structured ring spectra.
\newblock {\em Adv. Math.}, 298:654--692, 2016.

\bibitem{Ching_Harper_derived_Koszul_duality}
M.~Ching and J.~E. Harper.
\newblock Derived {K}oszul duality and {TQ}-homology completion of structured
  ring spectra.
\newblock {\em Adv. Math.}, 341:118--187, 2019.

\bibitem{Dror_Dwyer_long_homology}
E.~Dror and W.~G. Dwyer.
\newblock A long homology localization tower.
\newblock {\em Comment. Math. Helv.}, 52(2):185--210, 1977.

\bibitem{Dundas}
B.~I. Dundas.
\newblock Relative {$K$}-theory and topological cyclic homology.
\newblock {\em Acta Math.}, 179(2):223--242, 1997.

\bibitem{Dundas_Goodwillie_McCarthy}
B.~I. Dundas, T.~G. Goodwillie, and R.~McCarthy.
\newblock {\em The local structure of algebraic {K}-theory}, volume~18 of {\em
  Algebra and Applications}.
\newblock Springer-Verlag London, Ltd., London, 2013.

\bibitem{Dwyer_Miller_Neisendorfer}
W.~G. Dwyer, H.~R. Miller, and J.~Neisendorfer.
\newblock Fibrewise completion and unstable {A}dams spectral sequences.
\newblock {\em Israel J. Math.}, 66(1-3):160--178, 1989.

\bibitem{Dwyer_Spalinski}
W.~G. Dwyer and J.~Spali{\'n}ski.
\newblock Homotopy theories and model categories.
\newblock In {\em Handbook of algebraic topology}, pages 73--126.
  North-Holland, Amsterdam, 1995.

\bibitem{Goerss_Hopkins_moduli_problems}
P.~G. Goerss and M.~J. Hopkins.
\newblock Moduli problems for structured ring spectra.
\newblock 2005.
\newblock Available at \verb=http://hopf.math.purdue.edu=.

\bibitem{Goodwillie_calculus_2}
T.~G. Goodwillie.
\newblock Calculus. {II}. {A}nalytic functors.
\newblock {\em $K$-Theory}, 5(4):295--332, 1991/92.

\bibitem{Goodwillie_calculus_3}
T.~G. Goodwillie.
\newblock Calculus. {III}. {T}aylor series.
\newblock {\em Geom. Topol.}, 7:645--711 (electronic), 2003.

\bibitem{Harper_Zhang}
J.~E. Harper and Y.~Zhang.
\newblock Topological {Q}uillen localization of structured ring spectra.
\newblock {\em Tbilisi Math. J.}, 12(3):69--91, 2019.

\bibitem{Hirschhorn_overcategories}
P.~S. Hirschhorn.
\newblock Overcategories and undercategories of model categories.
\newblock {\em \\ \verb=arXiv:1507.01624 [math.AT]=}, 2015.

\bibitem{Hovey}
M.~Hovey.
\newblock {\em Model categories}, volume~63 of {\em Mathematical Surveys and
  Monographs}.
\newblock American Mathematical Society, Providence, RI, 1999.

\bibitem{Hovey_spectra}
M.~Hovey.
\newblock Spectra and symmetric spectra in general model categories.
\newblock {\em J. Pure Appl. Algebra}, 165(1):63--127, 2001.

\bibitem{Hovey_Shipley_Smith}
M.~Hovey, B.~Shipley, and J.~H. Smith.
\newblock Symmetric spectra.
\newblock {\em J. Amer. Math. Soc.}, 13(1):149--208, 2000.

\bibitem{Kuhn_survey}
N.~J. Kuhn.
\newblock Goodwillie towers and chromatic homotopy: an overview.
\newblock In {\em Proceedings of the {N}ishida {F}est ({K}inosaki 2003)},
  volume~10 of {\em Geom. Topol. Monogr.}, pages 245--279. Geom. Topol. Publ.,
  Coventry, 2007.

\bibitem{MacLane_categories}
S.~Mac~Lane.
\newblock {\em Categories for the working mathematician}, volume~5 of {\em
  Graduate Texts in Mathematics}.
\newblock Springer-Verlag, New York, second edition, 1998.

\bibitem{Quillen}
D.~Quillen.
\newblock {\em Homotopical algebra}.
\newblock Lecture Notes in Mathematics, No. 43. Springer-Verlag, Berlin, 1967.

\bibitem{Schonsheck_fibration}
N.~Schonsheck.
\newblock Fibration theorems for {TQ}-completion of structured ring spectra.
\newblock {\em Tbilisi Math. J.: Special Issue on Homotopy Theory, Spectra, and
  Structured Ring Spectra}, pages 1--15, 2020.

\bibitem{Schonsheck_TQ}
N.~Schonsheck.
\newblock {TQ}-completion and the {T}aylor tower of the identity functor.
\newblock {\em J. Homotopy Relat. Struct.}, 17(2):201--216, 2022.

\bibitem{Schwede_cotangent}
S.~Schwede.
\newblock Spectra in model categories and applications to the algebraic
  cotangent complex.
\newblock {\em J. Pure Appl. Algebra}, 120(1):77--104, 1997.

\bibitem{Schwede_book_project}
S.~Schwede.
\newblock {\em An untitled book project about symmetric spectra}.
\newblock 2007, 2009.
\newblock Available at:\\ \verb=http://www.math.uni-bonn.de/people/schwede/=.

\end{thebibliography}

\end{document}